\documentclass[12pt]{amsart}
\usepackage{amssymb,amsmath,latexsym, amscd, graphicx, url}
\usepackage{amsmath}
\usepackage{ amssymb }
\usepackage{ tipa }
\usepackage{tikz}
\usepackage{tikz-cd}
\usepackage{centernot}
\usepackage{mathtools}
\usepackage{ stmaryrd }
\usepackage{ mathrsfs }
\usepackage{fullpage}
\usepackage{hyperref}
\usepackage{cancel}
\usepackage{leftidx}
\usepackage{graphicx}
\usepackage{ upgreek }
\usepackage{graphicx, subcaption}
\newcommand*\tageq{\refstepcounter{equation}\tag{\theequation}}

\usetikzlibrary{matrix,arrows}
\hypersetup{colorlinks=true, linktoc=all, linkcolor=blue}

\newtheorem{theorem}{Theorem}[section]
\newtheorem{lemma}[theorem]{Lemma}

\newtheorem{proposition}[theorem]{Proposition}
\newtheorem{definition}[theorem]{Definition}

\newtheorem{assumptions} [theorem]{Assumptions}
\newtheorem{construction}[theorem]{Construction}
\theoremstyle{definition}
\newtheorem{remark}[theorem]{Remark}
\newtheorem{example}[theorem]{Example}

    \counterwithin{figure}{section}
\title[Cluster pictures for Hitchin fibers of rank two Higgs bundles]{Cluster pictures for Hitchin fibers of rank two Higgs bundles  \\ \vspace*{10pt} }
\author{Sina Zabanfahm} 
\begin{document}
\maketitle
\begin{abstract}
    Let $\varphi\colon X\rightarrow Y$ be a degree two Galois cover of smooth curves over a local field $F$ of odd characteristic. Assuming that $Y$ has good reduction, we describe a semi-stability criterion for the curve $X$, using the data of the branch locus of the covering $\varphi$. In the case that $X$ has semi-stable reduction, we describe the dual graph of the minimal regular model of $X$ over $F.$ We do this by adopting the notion of cluster picture defined for hyperelliptic curves for the case where $Y$ is not necessarily a rational curve. Using these results, we describe the variation of the p-adic volume of Hitchin fibers over the semi-stable locus of the moduli space of rank 2 twisted Higgs bundles.        
\end{abstract}
\tableofcontents
\section{Introduction}
    Let $F$ be a local field of odd characteristic and let $v\colon \mathcal{O}_F\rightarrow \mathbb{Z}$ denote the normalized valuation of the corresponding ring of integers. Assume that $\varphi\colon X\rightarrow Y$ is a degree two Galois cover of smooth curves over $F$, with the property that $Y$ has positive genus and has good reduction over $F$. Throughout this paper, we study some of the arithmetic properties of the curve $X$ by considering the cluster picture associated to the cover $\varphi\colon X\rightarrow Y$. In particular, we determine the reduction type of $X$ over $F$, and in the case that $X$ has  semi-stable reduction, we describe the dual graph of the minimal regular model of $X$ over a maximally unramified extension $F^{ur}/F$. In the case that $X$ is semi-stable, after describing the action of the absolute Galois group $G_F$ on the irreducible components of the minimal regular model of $X_{F^{ur}}$, we then compute the Tamagawa number of $Jac(X).$ For the case that $Y$ is a rational curve, these properties are described in the work of Tim and Vladimir Dokchitser, C\'eline Maistret and Adam Morgan in \cite{DDMM}.\par
    Let $C/F$ be a hyperelliptic curve given by Weierstrass equation
    \[
        C:\quad y^2=f(x)=c_f\prod_{r\in \mathcal{R}} (x-r).
    \]
    where $\mathcal{R}\subset \overline{F}$ denotes the set of roots of $f(x)$ over an algebraic closure $\overline{F}/F$ and $c_f$ is the leading coefficient of $f(x).$ The cluster picture associated to this hyperelliptic curve, as given in \cite{DDMM} is a collection of nested subsets of $\mathcal{R}$, which describes the $p-$adic distances of roots in $\mathcal{R}.$ The combinatorial data of this cluster picture, together with the valuation of the leading coefficient and the action of the Galois group, determines the reduction type of $C$ (\cite[Theorem 1.8]{DDMM}). Moreover, in the case that $C$ has semi-stable reduction, the combinatorial data captured by the cluster picture, gives a description of the dual graph of the minimal regular model of $C$ (\cite[Propostion 5.12]{DDMM}).\par
    In this paper, we consider degree two covers $\varphi\colon X\rightarrow Y$, where $Y$ admits a proper smooth model and has non-zero genus.
    Note that the reduction type of curves are not affected by unramified extension. Hence, in Section \ref{Subsection:semi-stabilty-criterion} we assume that $F$ is maximally unramified. In particular, the residue field of $\mathcal{O}_F$ is algebraically closed. The semi-stability criterion for $X$ is given by using the notion of the cluster picture associated to a degree two covering. Let $B\subset Y$ denote the branch locus of the covering $\varphi$ and let $L=F(B)$, be the compositum of the residue fields of closed points in $B.$ In the lack of defining equation for $X$ over $Y$, the cluster picture associated to this cover is a nested collection of subsets of $B(L).$ Each point in $B(L)$, can be viewed as a horizontal divisor inside the minimal regular model of $Y$. We denote the regular model of $Y/F$ by $\mathcal{Y}^{min}.$ These horizontal divisors intersect the special fiber of $\mathcal{Y}^{min}$ at closed points. The cluster picture associated to $\varphi$, determines the number of minimal blowups necessary for separating these horizontal divisors. In the case that $X$ is semi-stable, we describe the special fiber of $X$ and consequently we describe the dual graph of the minimal regular model of $X.$\par
    Assuming that $X$ is semi-stable, in Section \ref{Section: Galois action}, we provide an algorithm for computing the Tamagawa number of $Jac(X)$. To do this, we follow the approach of A. Betts as given in \cite{BETTS}. We briefly summarize this approach here. Let $Gr(X)$ denote the dual graph of $X$. Denote the first graph homology of $Gr(X)$ by $\Lambda:=H_1(Gr(X),\mathbb{Z})$. The length pairing on the edges of $Gr(X)$ induces an injection $\Lambda\hookrightarrow \Lambda^{\vee}.$ Define $\Upphi_{Gr(X)}:=\Lambda^{\vee}/\Lambda$. The action of $G(F^{ur}/F)$ on the minimal regular model of $X$ over $F^{ur}$ induces an action $\Upphi_{Gr(X)}$. The Tamagawa number of $Jac(X)$ over $F$ is given by
    \begin{align}\label{Formula: Tamagawa number}
       \Upphi_{Gr(X),F}:= \vert \Upphi_{Gr(X)}^{Gal(F^{ur}/F)}\vert,
    \end{align}
    Where $\Upphi_{Gr(X)}^{Gal(F^{ur}/F)}$ is the set of $Gal(F^{ur}/F)$ invariant elements of $\Upphi_{Gr(X)}$. In section \ref{Section: Galois action}, we describe the action of $Gal(F^{ur}/F)$ on $Gr(X).$\par
    By Lemma \ref{Lemma:2:1-covers-are-spectral}, degree two covers of schemes over $F$ or $\mathcal{O}_F$, can be viewed as spectral covers (Definition \ref{Definition: Spectral-covers}). In particular, given that $\mathcal{X}$ is the normalization of $\mathcal{Y}^{min}_s$ over $K(X),$ there exists a line bundle $\mathcal{L}\in Pic(\mathcal{Y})$ and a section $s\in \mathcal{L}^{\otimes 2}$, such that  
    \begin{align*}
        \mathcal{X}\cong \underline {Spec}(Sym^{\bullet} \mathcal{L}^\vee /(y^2-s)).    
    \end{align*}
    We define $v_{\varphi}$ to be the order of vanishing of $s$ on the irreducible components of $\mathcal{Y}^{min}_s.$ 
    Similar to the case of hyperelliptic curves, the reduction type of $X$ is determined by the combinatorial data of the cluster picture, the action of Galois group $G_F$ on the cluster picture and the numerical invariant $v_\varphi$ (See Theorem \ref{Theorem: Semi-stability-criterion-for 2:1-covers} and Proposition \ref{Proposition:Reduction-type}). In the case that $X$ has semi-stable reduction, the same data determines the description of the minimal regular model of $X$ over $F$ (See Proposition \ref{Proposition:Description-of-minimal-regular-model}).\par
In Section \ref{Section: Galois action}, we provide an algorithm for computing the normalized p-adic volume of Hitchin fibers over the semi-stable locus of the moduli space of rank 2 twisted Higgs bundles. In Section \ref{Section: Galois action}, we assume that $F$ has finite residue field and $Y/F$ is a smooth curve, which has good reduction over $F.$ Furthermore, we fix a line bundel $\mathcal{L}\in Pic(\mathcal{Y}^{min}).$ Let
        \begin{align*}
                    h:\mathcal{M}_{\mathcal{Y}}^\mathcal{L}&(2,d)\longrightarrow \mathcal{A}_{\mathcal{Y}}^\mathcal{L}(2,d)\\
                &(E,\theta)\longmapsto (Tr(\theta),det(\theta))
    \end{align*}
    denote the Hitchin map. For generic values of $\alpha \in \mathcal{A}_{\mathcal{Y}}^\mathcal{L}(2,d)$, by spectral correspondence, we have
    \begin{align*}
        h^{-1}(\alpha)=Jac(\mathcal{Y}_\alpha),
    \end{align*}
    where $\mathcal{Y}_\alpha$ is the spectral cover of $\mathcal{Y}^{min}$, associated to the collection of sections $\alpha.$ By Lemma \ref{Lemma: Normalized-volume}, computation of the normalized $p$-adic volume of $h^{-1}(\alpha)$, amounts in computing the Tamagawa number of $Jac(Y_{\alpha})$. We compute the latter, by applying Equation \ref{Formula: Tamagawa number} to $Y_{\alpha}/F.$
    \\
    \\
    \section{Preliminaries}\label{Section: Preliminaries}
    \subsection{Model of curves}\label{Subsection: Model of curves} Throughout this section, we assume that $F$ is a discrete valuation field with its ring of integers denoted by $\mathcal{O}_F$ and normalized valuation $v_F\colon \mathcal{O}_F\rightarrow \mathbb{Z}$. We denote the maximal ideal of $\mathcal{O}_F$ by $m_F=(\pi_F)$ and its residue field $\mathcal{O}_F/m_F$ by $k_F$. We always assume that $char(k_F)$ is distinct from 2. By a curve $X/F$ we mean a smooth, projective and geometrically connected scheme of dimension 1.
    \begin{definition}
        A model of a curve $X/F$ is a projective integral scheme $\mathcal{X}$ flat over $Spec{(\mathcal{O}_F)}$, together with an isomorphism $\mathcal{X}_F:=\mathcal{X}\times_{\mathcal{O}_F} Spec(F)\cong X$. A model of $\mathcal{X}$ of X is said to be regular (resp. normal) model of $X$, if $\mathcal{X}$ is a regular (resp. normal) scheme over $Spec(\mathcal{O}_F$).
    \end{definition}
    We denote curves over $Spec(F)$ by capital letters (for example $X,Y,Z,\dots$) and we denote arithmetic surfaces over $Spec(\mathcal{O}_F$) by calligraphic letters (for example $\mathcal{X},\mathcal{Y},\mathcal{Z},\dots$). Given $\mathcal{X}$ a model of $X$, we denote its special fiber by $\mathcal{X}_s:=\mathcal{X}\times_{Spec({\mathcal{O}_F})}Spec{(k_F)}$. When $\mathcal{X}$ is normal, then for any generic point $\xi\in \mathcal{X}_s$, $\mathcal{O}_{\mathcal{X},\xi}$ is a discrete valuation ring dominating $\mathcal{O}_F.$ We denote the valuation of $\mathcal{O}_{\mathcal{X},\xi}$ extending $v_F$ by $v_{\mathcal{\xi}}\colon K(X)\rightarrow \mathbb{Q}.$ Moreover, we denote the ramification index of the extension $\mathcal{O}_{\mathcal{X},\xi}/\mathcal{O}_F$ by $e(v_\xi,v).$
    \begin{lemma}
        Let $X/F$ be a smooth curve and $\mathcal{X}$ a normal model of $X$. Let $C\subset \mathcal{X}_s$ be an irreducible component of $\mathcal{X}_s$ with generic point $\mathcal{\xi}$. Then, $C$ is reduced if and only if $e(v_{\xi},v)=1$.
    \end{lemma}
    \begin{proof}
        Assume that $\mathcal{X}_s=\sum n_i C_i$, where the sum ranges over the irreducible components of $\mathcal{X}_s$ and $n_i$ is the multiplicity of $C_i$ in $\mathcal{X}_s.$ Then, the irreducible component $C$ is reduced if and only if it has multiplicity 1 in $\mathcal{X}_s$. Given a uniformizer $\pi_F\in \mathcal{O}_F$, we have that the multiplicity of $C$ is given by $v_\xi(\pi_F)=e(v_\xi,v).$
    \end{proof}
    \begin{definition}
        A normal model $\mathcal{X}/\mathcal{O}_F$ of $X$ is semi-stable, if the geometric fiber of $\mathcal{X}$ is reduced with at worst nodal singularities. A curve $X/F$ is semi-stable if it admits a semi-stable model over $\mathcal{O}_F$.
    \end{definition}
    By a theorem of Deligne and Mumford, given a curve $X/F$, there exists a finite extension $L/F$ such that $X_L$ has a  semi-stable model (\cite{DM}). However, their proof of this theorem is not constructive, it provides no information on how to construct such a field extension $L/F.$ A more constructive proof of this theorem is provided in \cite{AW}, which provides an extension $L/F$ where $X$ becomes semi-stable over $L$, although $L/F$ is not necessarily minimal extension with this property. Generally, in the case that $\mathcal{X}$ has potentially tame semi-stable reduction, the combinatorics of the dual graph of a good model (see \cite[Section 1.8]{LL99}) of $X$ can be used for determining a field extension $L/F$ where $X$ becomes semi-stable over $F.$ In general, these extensions are not minimal. On the other hand, assuming that $X$ only becomes semi-stable over a wildly ramified extension, the combinatorics of the dual graph of a regular model of $X$ is not enough for determining a field extension $L/F$, where $X$ becomes semi-stable over $L$.
    \begin{remark}
        When $k_F$ is algebraically closed, a singular point $x\in \mathcal{X}_s$ is nodal if and only if $\mathcal{O}_{\mathcal{X},x}\cong \mathcal{O}_F[[u,v]]/(uv-\pi_F^r)$ for some positive integer $r$. A point $x\in \mathcal{X}_s$ is a smooth point if and only if $\mathcal{O}_{\mathcal{X},x}\cong \mathcal{O}_F[[u]].$
    \end{remark}
    \begin{remark}
        Assuming that $g(X)>0$, up to isomorphism, there exists a unique regular model denoted by $\mathcal{X}^{min}$ dominated by any other regular models. If $X$ has semi-stable reduction, then $\mathcal{X}^{min}$ is a semi-stable model of $X$.
    \end{remark}
    \begin{definition}\label{Definition: Galois covering of smooth curves}
        A finite covering of smooth curves $\varphi\colon X\rightarrow Y$ over $F$ is a Galois covering with the Galois group $G$, if $Y\cong X/G$ and $K(X)/K(Y)$ is Galois with the Galois group $Gal(K(X)/K(Y))=G.$
    \end{definition}
    Next, we are going to show that degree two Galois cover of schemes can be viewed as spectral covers.
    \begin{definition}\label{Definition: Trace map}
        Let $Y$ be a scheme and let $\mathcal{F}$ be a locally free sheaf of commutative $\mathcal{O}_Y$ algebras. By \[Tr\colon \mathcal{F}\rightarrow \mathcal{O}_Y,\]
        we mean the trace map of $\mathcal{F}$ which is locally defined as follow. Over a trivializing open set $U\subset Y$, define
        \begin{align*}
            Tr(U)\colon&\mathcal{F}(U)\rightarrow \mathcal{O}_Y(U)\\
            &a\mapsto Tr(l_a),
        \end{align*}
        where $l_a\in End_{\mathcal{O}_Y(U)}(\mathcal{F}(U),\mathcal{F}(U))$ is the endomorphism given by left multiplication by $a$, and $T r:End_{\mathcal{O}_Y(U)}(\mathcal{F}(U),\mathcal{F}(U))\rightarrow \mathcal{O}_Y(U)$ is the usual trace map defined on the endomorphism ring $End_{\mathcal{O}_Y(U)}(\mathcal{F}(U),\mathcal{F}(U))$.
    \end{definition}
    \begin{definition}\label{Definition: Spectral-covers}
        Let $L$ be a line bundle on the scheme $Y$. Let $\pi\colon TotL\rightarrow Y$ denote the projection map. The line bundle $\pi^*L$ over $TotL$ has an obvious global section, which we denote it by $y\in \pi^*L(TotL).$\par
        Given a collection of sections $\alpha:=\{\alpha_i \in H^0(Y,L^{\otimes i})\}_{1 \leq i\leq n }$, the spectral cover associated to this collection of sections is the subscheme $Y_{\alpha}:=\underline{Spec}_Y(L^{\vee}/y^n- \alpha_1 y^{n-1} + \dots + (-1)^n\alpha_n)$ of $TotL$ given by the zeros of the section
        \begin{align*}
            y^n- \alpha_1 y^{n-1} + \dots + (-1)^n\alpha_n \in (\pi^*L)^{\otimes n}(TotL).
        \end{align*}
    \end{definition}
    \begin{lemma}\label{Lemma:2:1-covers-are-spectral}
        Let $\varphi\colon X\rightarrow Y$ be a finite and flat morphism of degree 2 of schemes, and assume that 2 is invertible in $\mathcal{O}_Y^{\times}$. Then, there exists a line bundle $L\in Pic(Y)$ and a section $s\in L^{\otimes 2}$ such that $X$ is isomorphic to $\underline{Spec}(Sym^\bullet L^\vee/(y^2-s)).$
    \end{lemma}
    \begin{proof}
        As $\varphi$ is an affine morphism and $Y$ is locally Noetherian, the category of degree 2 covers of $Y$ is equivalent to the rank two locally free sheaves of $\mathcal{O}_Y$-algebras. Therefore, the data of a degree 2 covering $X$ of $Y$ is equivalent to the sheafified spectrum of a rank two locally free sheaf of commutative algebras  $\varphi_*\mathcal{O}_X$.\par
        Let $Tr\colon \varphi_*\mathcal{O}_X\rightarrow \mathcal{O}_Y$ denote the trace map. Since the composition $\mathcal{O}_Y\rightarrow \varphi_*\mathcal{O}_X\xrightarrow{\frac{1}{2}Tr}\mathcal{O}_Y$ is equal to $id_{\mathcal{O}_Y}$, we get that $\varphi_*\mathcal{O}_X\cong M\oplus \mathcal{O}_Y$, for a line bundle $M\in Pic(Y)$. Set $L:=M^{\vee}.$ Define $s\in L^{\otimes 2}(Y)$ to be the section defined by composition 
        \begin{align*}
            M^{\otimes{2}}\hookrightarrow(\varphi_*\mathcal{O}_X)^{\otimes 2}\rightarrow \varphi_*\mathcal{O}_X\xrightarrow{\frac{1}{2}Tr}\mathcal{O}_Y,
        \end{align*}
        where the second map is the multiplication map on $\varphi_*\mathcal{O}_X.$ We claim that 
        \begin{align*}
            \underline{Spec}_Y(Sym^\bullet L^\vee/(y^2-s))\cong \underline{{Spec}}_Y(\varphi_*\mathcal{O}_X).
        \end{align*}
        To verify this, it suffices to show that $Sym^\bullet L^\vee/(y^2-s))\cong \varphi_*\mathcal{O}_X$ as free $\mathcal{O}_Y$ algebras. As $Sym^\bullet L^\vee/(y^2-s)\cong L^{\vee}\oplus \mathcal{O}_Y$ we just need to verify that the multiplication map on these two $\mathcal{O}_Y$ algebras are the same.
        Let $U\subset Y$ be an affine open subset. Then, over this open set, the multiplication map on $Sym^\bullet L^\vee/(y^2-s)$ is given by:
        \begin{align*}
            (L^{\vee}\vert_{U}\oplus \mathcal{O}_U)&\otimes (L^{\vee}\vert_U \otimes \mathcal{O}_U)\rightarrow L^{\vee}\vert_U\oplus \mathcal{O}_U,\\
            &(l_1,r_1) \otimes (l_2,r_2)\mapsto (r_1l_2+r_2l_1,s(l_1\otimes l_2)+ r_1r_2),
        \end{align*}
            which is the same as the multiplication map on $\varphi_*\mathcal{O}_Y.$
    \end{proof}
    We fix $\varphi\colon X\rightarrow Y$ to be a degree 2 Galois covering with Galois group $G$. Given that $\mathcal{Y}/\mathcal{O}_F$ is a normal model of $Y$, we denote the normalization of $\mathcal{Y}$ in $K(X)$ by $N(\mathcal{Y},K(X))$. Note that the morphism $\varphi\colon X\rightarrow Y$ extends uniquely to a morphism of normal models, which by an abuse of notation, we also denote it by $\varphi\colon N(\mathcal{Y},K(X))\rightarrow \mathcal{Y}.$ As $\mathcal{O}_F$ is an excellent ring, the induced morphism of normal models is finite. Moreover, assuming that $\mathcal{Y}$ is regular, we get that the induced morphism $\varphi:\mathcal{X}\rightarrow \mathcal{Y}$ is flat (\cite[18.H]{Mats-Comalg}) .  \par
    Since $\varphi$ is a degree two Galois cover, by Lemma \ref{Lemma:2:1-covers-are-spectral}, there exists a line bundle $\mathcal{L}\in Pic(\mathcal{Y})$ and a section $s\in \mathcal{L}^{\otimes2}(\mathcal{Y})$, such that $\underline{Spec}_{{\mathcal{Y}}}Sym^\bullet\mathcal{L}^\vee/(y^2-s)$ is isomorphic to $N(\mathcal{Y},K(X)).$ We denote the scheme $\underline{Spec}_{{\mathcal{Y}}}Sym^\bullet\mathcal{L}^\vee/(y^2-s)$ by $S(\mathcal{Y},\mathcal{L},s)$.
    \begin{remark}\label{Remark:TheoremofRaynaud}
        Assuming that $X$ is semi-stable, a theorem of Raynaud(\cite{Raynaud}) gives that $Y$ is also semi-stable. Moreover, considering the minimal model of $X$, the action of the Galois group $G$ on $X$ extends to an action on $\mathcal{X}^{min}$ and $\mathcal{X}^{min}/G$ is a semi-stable model of $Y$. However, this quotient is not necessarily regular. In this case, a sequence of blowups centered singular points of the quotient gives a regular semi-stable model of $X$ (\cite[Proposition 10.3.48]{Liu-Alg-Geom}).
    \end{remark}
    For the remaining part of this subsection, we assume that $\varphi\colon X\rightarrow Y$ is a degree two Galois cover of curves and $\mathcal{Y}/\mathcal{O}_F$ is a normal model of $Y$ and we set $\mathcal{X}:=N(\mathcal{Y},K(X))$. We fix $\mathcal{L}\in Pic(\mathcal{Y})$ and $s\in \mathcal{L}^{\otimes 2}(\mathcal{Y})$ such that $N(\mathcal{Y},K(X))=S(\mathcal{Y},\mathcal{L},s).$
    \begin{lemma}\label{Lemma:Trivialization-of-the-covering}
        Let $C\subset \mathcal{Y}_s$ be a reduced irreducible component of $\mathcal{Y}_s$, with the generic point $\xi\in C$. Then,
        \begin{enumerate}
            \item $\mathcal{O}_{\mathcal{Y},\xi}\otimes \mathcal{O}_\mathcal{X}\cong \mathcal{O}_{\mathcal{Y},\xi}[z]/(z^2-\alpha)$ for some $\alpha \in \mathcal{O}_{\mathcal{Y},\xi}$ unique up to a multiplication by a unit in $\mathcal{O}_{\mathcal{Y},\xi}.$
            \item There are two irreducible components contained in $\varphi^{-1}(C)$ if and only if $\alpha$ is a non-zero square in $k(\xi).$ \item An irreducible component contained in $\varphi^{-1}(C)$ is reduced if and only if $v_\xi(\alpha)$ is even.
        \end{enumerate}
    \end{lemma}
    \begin{proof}
        For any affine open set $\xi \in Spec(A)\subset \mathcal{Y}$, we have:
        \begin{align*}
            Spec(A)\otimes \mathcal{O}_\mathcal{X}=A[z]/(z^2-a)
        \end{align*}
        for some $a \in A.$
        Therefore:
        \begin{align*}
            \mathcal{O}_{\mathcal{Y},\xi}\otimes \mathcal{O}_{\mathcal{X}}&=\lim_{\longrightarrow}A_i\otimes \mathcal{O}_\mathcal{X}\\
            &=\lim_{\longrightarrow}(A_i\otimes\mathcal{O}_{\mathcal{X}})\\  &=\lim_{\longrightarrow}A_i[z]/(z^2-a\vert_{Spec(A_i)})\\ &=\mathcal{O}_{\mathcal{Y},\xi}[z]/(z^2-\alpha),
        \end{align*}
        for some $\alpha \in \mathcal{O}_{\mathcal{Y},\xi}$, where the limit is taken over all the affine open set $Spec(A_i)\subset Spec(A)$ containing the point $\xi.$\par
        There are two irreducible components contained in $\varphi^{-1}(C)$ if and only if $\varphi^{-1}(\xi)$ consists of two generic points of $\mathcal{X}$. The latter statement is equivalent to the existence of the following isomorphism 
        \begin{align*}
            \mathcal{O}_{\mathcal{Y},\xi} [z]/(z^2-\alpha)\otimes_{\mathcal{O}_{\mathcal{Y}.\xi}}k(\xi)\cong k(\xi)\oplus k(\xi),
        \end{align*}
        Note that such isomorphism exists if and only if $\alpha$ is a square in $k(\xi).$ \par
        Let $D\subset \varphi^{-1}(C)$ be an irreducible component of $\mathcal{X}_s$ and let $\eta$ be the generic point of $D$. If $\alpha$ is a non-zero square in $k(\xi)$, by the previous part there are two points in the pre-image of $\xi$, each of which have to be unramified over $\xi$ and as $C$ is reduced, we get that the corresponding two irreducible components in $\varphi^{-1}(C)$ are reduced. As $\alpha$ is a nonzero square in $k(\xi)$, we get that $\alpha\in \mathcal{O}_{\mathcal{Y},\xi}^{\times}$ and $v_{\xi}(\alpha)=0.$ Assuming that $\alpha$ is not a nonzero square in $k(\xi)$, then $\mathcal{O}_{\mathcal{X},\eta}\cong\mathcal{O}_{\mathcal{Y},\xi}[z]/(z^2-\alpha)$. In this case we have $e(v_\eta,v)=1$  if and only if $v_\xi(\alpha)$ is even.
    \end{proof}
    Note that the same computation as given in the proof of Lemma \ref{Lemma:Trivialization-of-the-covering} gives that for any point $y\in \mathcal{O}_\mathcal{Y}$, we have \begin{align}\label{Equation: Trivialization}    
        \mathcal{O}_{\mathcal{Y},y}\otimes \mathcal{O}_{\mathcal{X}}=\mathcal{O}_{\mathcal{Y},y}[z]/(z^2-\alpha_y),    
    \end{align}
    for some $\alpha_y\in \mathcal{O}_{\mathcal{Y},y}.$
    \begin{definition}\label{Definition:Order-of-vanishing}
        Given a point $y\in \mathcal{Y}_s$, we call $\alpha_y$ as given in Equation \ref{Equation: Trivialization}, the trivialization of the covering $\varphi\colon \mathcal{X}\rightarrow \mathcal{Y}$ at $y$. Given that $C\subset \mathcal{Y}_s$ is an irreducible component with the generic point $\xi$, we set $v_\varphi(C):=v_\xi(\alpha_\xi).$
    \end{definition}
    We can rephrase Lemma \ref{Lemma:Trivialization-of-the-covering} as follow.
    \begin{lemma}\label{Lemma:Trivialization-of-the-covering-using-vanishing-order}
        Let $C\subset \mathcal{Y}_s$ be an irreducible component of $\mathcal{Y}_s$ with the generic point $\xi.$ Then,
        \begin{enumerate}
            \item An irreducible component of $\mathcal{Y}_s$ contained in $\varphi^{-1}(C)$ is reduced if and only if $v_\varphi(C)$ is even.
            \item There are two irreducible components in $\varphi^{-1}(C)$ if and only if the trivialization of the covering $\varphi$ at $\xi$ is a square in $k(\xi)$.
        \end{enumerate}
    \end{lemma}
    Given a degree two or three Galois cover of smooth curves $\varphi\colon X\rightarrow Y$, by \cite[Theorem 7.3]{LL99}, it is possible to find a regular model $\mathcal{Y}/{\mathcal{O}_F}$ of $Y$, with the property that its normalization in $K(X)$ is a regular model of $X.$ Lemma \ref{Lemma:Simultaneous-desingularization} states a criterion for regular models of $Y$ with the property that their normalization in $K(X)$ is a regular model of $X$. 
    \begin{lemma}(\cite[Lemma 2.1]{Srin15})\label{Lemma:Simultaneous-desingularization}
        Assume that $div(s)=\sum v_{C} C$ where $C$ varies over the irreducible components of $\mathcal{Y}$. Moreover, assume that the following two conditions hold:
        \begin{enumerate}
            \item For any two irreducible component $C_i$ where $v_{\varphi}(C_i)$ is odd, they do not intersect.
            \item Any $C_i$ for which $v_{\varphi}(C_i)$ is odd is regular.
        \end{enumerate}
        Then, $\mathcal{X}$ is a regular model of $X$.
    \end{lemma}
    \begin{proof}
        This is  \cite[Lemma 2.1]{Srin15} written in terms of the divisor of a section in the line bundle $\mathcal{L}^{\otimes2}$, rather than a rational function. In particular, the same proof applies.
    \end{proof}
    \begin{lemma}
        Let $\varphi\colon X\rightarrow Y$ be a degree 2 Galois cover of smooth curves and let $\mathcal{Y}$ be a regular model of $Y$ where $\mathcal{X}=N(\mathcal{Y},K(X))$ is regular. Assume that $C\subset \mathcal{Y}_s$ is an irreducible component where $C.C=-1$ and $\varphi^{-1}(C)$ consists of two irreducible components where $D_1\cap D_2\neq \emptyset$. Then $D_1$ and $D_2$ are not exceptional.
    \end{lemma}
    \begin{proof}
        If $D_1$ is exceptional, then by \cite[Theorem 9.3.8]{Liu-Alg-Geom} $D_1.D_1=-1$. On the other hand by the projection formula \cite[Theorem 9.2.12]{Liu-Alg-Geom}, we see that $D_1.(D_1+D_2)=C_1.C_1=-1$. This implies that $D_1.D_2=0$ which contradicts the assumption that $D_1\cap D_2\neq \emptyset.$
    \end{proof}
    \begin{lemma}\label{Lemma: Extenstion-of-the-line-bundle-to-the-blowup-map}
        Let $\varphi\colon X\rightarrow Y$ be a degree 2 Galois cover of smooth curves over $F$ and let $\mathcal{Y}/\mathcal{O}_F$ be a regular model of $Y.$ Let $\pi\colon  \mathcal{Y}^{\prime}\rightarrow \mathcal{Y}$ be the blowup map centered at a closed smooth point $y\in \mathcal{Y}_s$. Given a line bundle $\mathcal{L}\in Pic(\mathcal{Y})$ and a section $s\in \mathcal{L}^{\otimes 2}(\mathcal{Y})$ such that \[N(\mathcal{Y},K(X))\cong S(\mathcal{Y},\mathcal{L},s).\] Then,
        \begin{align*}            
            N(\mathcal{Y}^{\prime},K(X))=S(\mathcal{Y}^{\prime},\pi^*\mathcal{L},\pi^*s).
        \end{align*}
    \end{lemma}
    \begin{proof}
            Let $U=Spec(A)\subset \mathcal{Y}$, be an affine neighborhood of $y$ where the line bundle $\mathcal{L}$ trivializes. Assume that the restriction of the section $s$ on $Spec(A)$ is given by $a\in A.$ Then, we get that $K(X)=Frac(A)(\sqrt{a}),$ and $\varphi^{-1}(U)=Spec(A[z]/(z^2-a))$. Let $m\subset A$ be the maximal ideal corresponding to the point $y$ and assume that this maximal ideal is generated by $\{f_1,\dots, f_n\}\subset A.$ \par 
            We can cover $\pi^{-1}(U)$ by open affine subsets $U_i:=Spec(A_{(f_i)}[f_i^{-1}f_j]_j)$ (\cite[Lemma 8.1.2]{Liu-Alg-Geom}). As $\pi^*(\mathcal{L})\vert_{\pi^{-1}(U_i)}$ is also trivial, $\pi^*{\mathcal{L}^{\otimes 2}}(U_i)=A_{(f_i)}[f_i^{-1}f_j]_j$. Moreover, the restriction of $s$ on $U_i$ is given by $a\in Spec(A)_{(f_i)}[f_i^{-1}f_j]_j$. Therefore, 
            \begin{align*}
                \varphi^{-1}(U_i)=Spec(
                A_{(f_i)}[f_i^{-1}f_j]_j[z]/(z^2-a)),
            \end{align*}
            which is the normalization of $U_i$ in $K(X).$
    \end{proof}
    \subsection{Moduli space of Higgs bundles}\label{Subsection:Prelimeneries on the moduli space of Higgs bundles}
    In this subsection we assume that $Y/F$ is a smooth curve, $\mathcal{Y}/\mathcal{O}_F$ a smooth proper model of $Y$ and $\mathcal{L}$ is a line bundle on $\mathcal{Y}.$
    \begin{definition}
        A rank $n$ $\mathcal{L}$-twisted Higgs bundle on $\mathcal{Y}$ with coefficients in $\mathcal{L}$, is a pair $(E,\theta)$, where $E$ is a rank $n$ vector bundle and $\theta$ is a global section of $End(E)\otimes \mathcal{L}.$ A Higgs bundle $(E,\theta)$ is (semi-)stable, if for any sub-bundle $F\subset E$, such that $\theta(F)\subset F\otimes \mathcal{L}$, then:
        \begin{align*}
            \frac{deg(F)}{rank(F)}&<\frac{deg(E)}{rank(E)}.\\
            (&\leq)
        \end{align*}
    \end{definition}
    \begin{definition}\label{Definition:Hitchin fibration}
        The Hitchin map is defined by
        \begin{align*}
            h\colon \mathcal{M}_{\mathcal{Y}}^\mathcal{L}&(n,d)\longrightarrow \mathcal{A}_{\mathcal{Y}}^\mathcal{L}(n,d):=\oplus_{i=1}^rH^{0}(\mathcal{Y},\mathcal{L}^{\otimes i})\\
            &(E,\theta)\longmapsto ((-1)^i Tr(\wedge^i\theta)).
        \end{align*}
        The affine space $\mathcal{A}_{\mathcal{Y}}^\mathcal{L}(n,d)$ is called the Hitchin base.
    \end{definition}
    \begin{remark}\label{Remark: Spectral correspondence}
        For a generic choice $\alpha=(\alpha_1,\dots,\alpha_n)\in \mathcal{A}_{\mathcal{Y}}^\mathcal{L}(n,d)$, we have
        \begin{align*}
            h^{-1}(\alpha)\cong Jac(\mathcal{Y}_\alpha),
        \end{align*}
        where $\mathcal{Y}_\alpha$ is the spectral cover of $\mathcal{Y}$ given by 
            \begin{align*}
                \mathcal{Y}_\alpha:=\underline{Spec}(Sym^\bullet\mathcal{L}^\vee/(y^n-\alpha_1y^{n-1}-\dots+ (-1)^n\alpha_n)).
            \end{align*}
    \end{remark}
    \begin{remark}\label{Remark: Rank 2 Higgs bundles}
        In the case of the rank two moduli space of Higgs bundles, the Hitchin map is given by
        \begin{align*}
            h\colon\mathcal{M}_{\mathcal{Y}}^\mathcal{L}&(2,d)\longrightarrow \mathcal{A}_{\mathcal{Y}}^\mathcal{L}(2,d)\\
                &(E,\theta)\longmapsto (Tr(\theta),det(\theta)).
        \end{align*}
        The spectral curve determined by $\alpha:=(\alpha_1,\alpha_2)$ is the closed subscheme of $Tot(\mathcal{L})$ given by $\mathcal{Y}_\alpha=(Sym^\bullet\mathcal{L}^\vee/(y^2-\alpha_1 y+ \alpha_2)).$
        By a change of variable we can always assume that $\alpha_1$ is zero. Note that $\mathcal{Y}_\alpha$ is a degree two cover of $\mathcal{Y}$. Moreover, $\mathcal{Y}_{\alpha}$ admits a natural involution action, quotient of which is isomorphic to $\mathcal{Y}.$
    \end{remark}
    \subsection{p-adic volume}
    Throughout this subsection we impose the extra condition that $k_F$ is a finite field. An $F$-analytic manifold is a Hausdorff and second countable topological space together with a choice of maximal atlas. In this setting one requires transition functions to be bi-analytic. The notion of top degree differential forms over an $F$-analytic manifold, is defined similarly as in the case of complex and real manifolds. Moreover, as $F^d$ is a locally compact topological group, it admits a Haar measure. Using this measure, one defines the notion of $p$-adic integration over $F$-analytic manifolds.\par  
    Given a smooth $\mathcal{O}_F$-variety $\mathcal{X}$, induced with its analytic topology, $\mathcal{X}(\mathcal{O}_F)$, has a natural $F$-analytic manifold structure. The following theorem of Weil relates the $p$-adic volume of this manifold, to the point count of its special fiber.
    \begin{theorem}(Weil                               1982,\cite{Weil})\label{weil-canonicalmeasure}
        The $F$-analytic manifold $\mathcal{X}(\mathcal{O}_F)$ admits a canonical measure $\mu_\mathcal{X}$ such that:
        \begin{align*}
            \mbox{Vol}_{\mu_\mathcal{X}}(\mathcal{X}(\mathcal{O}_F))=(\frac{1}{Card({k})})^d \mathcal{X}_s(k),
        \end{align*}
        where $d$ denotes the dimension of $\mathcal{X}_s:=\mathcal{X}\times_{\mathcal{O}_F}Spec{(k)}.$
    \end{theorem}
    Given an abelian variety, its $p$-adic volume can be understood by studying its N\'eron model (defined below).
    \begin{definition}\cite{BLR}
        Suppose $A/F$ is an abelian variety. Its N\'eron model $\mathcal{A}/\mathcal{O}_F$ is a smooth, separated and finite type group scheme with the universal property that given any smooth scheme $\mathcal{Y}/\mathcal{O}_F$, any morphism $\phi_F\colon \mathcal{Y}_F\rightarrow \mathcal{A}_F$ extends uniquely to $\phi\colon \mathcal{Y}\rightarrow \mathcal{A}. $
    \end{definition}
     As a direct consequence of the above definition, we see that $A(F)=\mathcal{A}(\mathcal{O}_F).$ The lemma given below, relates the $p$-adic volume of $A(F)$ to the Tamagawa number of its N\'eron model:
     \begin{lemma}(\cite{GM}, section 3.3)\label{Lemma: Normalized-volume}
        Suppose $\omega\in \Omega^{top}_{A/F}.$ This differential form induces a top degree differential form on $A(F)$ which we also denote by $\omega.$ Then, there exist a positive rational number $c_\omega(A)$ called the conductor, such that:
        \begin{align*}
            V_\omega(A):=\int_{A(F)}\lvert \omega \rvert = c_\omega(A).\frac{\lvert\mathcal{A}_s^0 \rvert.\lvert\Upphi_A(k_F) \rvert}{q^d}.
        \end{align*}
        Here, $\mathcal{A}^0_{s}$ denotes the connected component of the identity in $\mathcal{A}_s$ and  $\Upphi_A$ is a finite \'etale group scheme called the group of connected components of $\mathcal{A}^0_s.$
    \end{lemma}
    The size of $\Upphi_A(k_F)$ is called the Tamagawa number of $A$. By the previous lemma, computation of the $p$-adic volume of abelian varieties can be reduced to understanding their Tamagawa number.
    In the case that $A$ is the Jacobian of a semi-stable curve $Y/F$, $\vert \Upphi_{Jac(Y)}(k_F)\vert$ can be computed by
    studying the action of the absolute Galois group on the dual graph of the minimal regular model of $Y$. The main reference for this part is \cite{BETTS}. 
    \begin{definition}\label{Definition: Dual-Graph}
        Let $\mathcal{Y}^{min}/\mathcal{O}_F$ denote the minimal regular model of $Y$ over $F$. Let $\mathcal{Y}^{min}_s$ denote the special fiber of this model. Assume that $\{\Gamma_1,\dots ,\Gamma_n\}$ are the irreducible components of $(\mathcal{Y}^{min}_s)_{\overline{k_F}}$. The dual graph associated to $Y$, which denote by $Gr(Y)$ is defined as follow:
        \begin{itemize}
            \item Vertices of $Gr(Y)$ corresponds to the irreducible components of $(\mathcal{Y}^{min}_s)_{\overline{k_F}}$.
            \item given vertices $v_i,v_j\in V(Gr(X))$, the number of edges between $v_i$ and $v_j$ corresponds to the number of intersection points between $\Gamma_{i}$ and $\Gamma_{j}$.
        \end{itemize}
    \end{definition}
    \begin{remark}\label{Remark: Dual-graph using maximally unramified}
        Alternatively, one can define the dual graph $Gr(Y)$ by considering the minimal regular model of $Y$ over $F^{ur}$. Let $\mathcal{Y}/\mathcal{O}_{F^{ur}}$, denote the minimal regular model of $Y$ over $F^{ur},$ where $F^{ur}$ is a maximally unramified extension of $F$ inside a choice of an algebraic closure. Then, vertices of $Gr(Y)$ corresponds to the irreducible components of the special fiber $\mathcal{Y}_s$. Let $v_i$ $v_j$ be vertices in $Gr(Y)$, corresponding to the irreducible components $\Gamma_i,\Gamma_j\subset \mathcal{Y}_s$. Then, the number of edges between $v_i$ and $v_j$ is given by the number of the points in the intersection $\Gamma_i\cap \Gamma_j.$ Note that $Gr(Y)$ is a connected graph, since $\mathcal{Y}_s$ is connected. It is possible for $Gr(Y)$ to have loops (edges with the same endpoints). However, we assume that $Gr(Y)$ is loopless, as deleting loops will not effect the Tamagawa number of the graph $Gr(Y).$
    \end{remark}
    The absolute Galois group $G_F$ acts on $(\mathcal{Y}^{min}_s)_{\overline{k_F}}$, which induces an action on $Gr(X)$. This action, determines the Tamagawa number of $Jac(Y).$ Note that this action is unramified, and hence it is determined by a choice of Frobenius $Frob\in G_F$. In the remaining of this section, we summarize how to compute the Tamagawa number of $\vert \upphi_{Jac(Y)}(k_F) \vert$, using the action of $Frob$ on $Gr(Y)$.\par
    Let $\Lambda:=H_1(Gr(Y),\mathbb{Z})$ denote the first homology group of graph $Gr(Y)$ (\cite[Definition 2.1.1]{BETTS}). Fixing an orintation on $Gr(Y)$, gives an integer valued intersection pairing on $\Lambda$, which in turn induces an injection $\Lambda\hookrightarrow \Lambda^\vee$ (see \cite[2.1.1]{BETTS}, or \cite[Section 2.2]{DDMM2}). The jacobian of the graph $Gr(Y)$, is defined by $Jac_{Gr(Y)}=\Lambda^{\vee}/\Lambda.$ This construction is independent of the choice of orientation and it is functorial with respect to graph isomorphisms (\cite[Proposition 2.1.3]{BETTS}). In particular, the action of $Frob$ on $Gr(Y)$, induces an action on $Jac_{Gr(Y)}$. By \cite[Theorem 2.1.8]{BETTS}, we have the following equality:
    \begin{align}\label{Equation: Computation of Tamagawa number}
        \vert \upphi_{Jac(Y)}(k_F)\vert = \vert Jac_{Gr(Y)}^{Frob}\vert,
    \end{align}
    where $Jac_{Gr(Y)}^{Frob}$ is the subgroup of $Frob$ invariant elements of $Jac_{Gr(Y)}.$
    \begin{remark}\cite{LP}
        Assume that $F=F^{ur}$ and $X/F$ is a smooth curve that has semi-stable reduction over $F$. Let $\mathcal{Y}/\mathcal{O}_F$ be the minimal regular model $X$. Then,
        \begin{align*}
            \vert \Upphi_{Jac(X)}(k_F)\vert = \vert Jac(Gr(Y))\vert,
        \end{align*}    
        which is equal to the number of maximal spanning trees of $Gr(Y).$
    \end{remark}
    \section{Semi-stability criterion}
     Throughout this section, we will always work with the assumptions and notations given in Assumption
     \ref{Assumption:Main assumption}.
    \begin{assumptions}\label{Assumption:Main assumption}
        Let $\varphi\colon X\rightarrow Y$ be a degree 2 Galois cover of smooth curves. Moreover, assume that $g(Y)>0$ and $Y$ admits a proper smooth model over $Spec(\mathcal{O}_F)$, which we denote by $\mathcal{Y}^{min}$. We let $B\subset Y$ denote the branch locus of this covering and set $L:=F(B)$, where by $F(B)$ we mean the compositum of residue fields of points in $B.$ We denote the degree of this field extension by $e:=[L:F].$ We also assume that $F=F^{ur}$, where by $F^{ur}$ we mean the maximally unramified field extension of $F$ inside an algebraic closure of $F$. We can impose this assumption as the reduction type of a curve is not affected by unramified extensions.
    \end{assumptions}   
    \subsection{Cluster pictures}
    In this subsection, we introduce the notion of the cluster picture assoicated to the covering $\varphi\colon X\rightarrow Y$. Cluster pictures for hyperelliptic curves are introduced in \cite{DDMM} using Weierstrass equations. The definition in this subsection, is an attempt to capture the properties of the cluster picture as given in \cite{DDMM}, for the case where  $Y$ is not necessarily a rational curve. Even in the case where $Y$ is a rational curve, our notion of cluster picture slightly differs from the one given in loc. cit. (see Example \ref{Example: Difference between the notion of cluster pictures}). For example, the data of the branch point at infinity is evident in the depiction of our cluster picture. However, these differences are formal, and both notions of cluster picture for hyperelliptic curves are equivalent. \par
    Consider the following sequence of maps:
    \[
        \begin{tikzcd}
            B(\mathcal{O}_L/m_L) & B(\mathcal{O}_L/m_L^2) \arrow[l]                  & \dots \arrow[l] & B(\mathcal{O}_L/m_L^n) \arrow[l] & \dots \\
                     & B(\mathcal{O}_L) \arrow[lu] \arrow[u] \arrow[rru] &                 &                                  &       \\
                     & {B(L),} \arrow[u, "\cong"]                        &                 &                                  &      
        \end{tikzcd}
    \]
    where the morphism $B(L)\rightarrow B(\mathcal{O}_L)$ is the isomorphism provided by the valuative criterion of properness. For $i\geq 1$ We denote the map from $B(\mathcal{O}_L)$ to $B(\mathcal{O}_L/{m_L^i})$, by 
    \begin{align*}
        p_i\colon B(\mathcal{O}_L)\rightarrow  B(\mathcal{O}_L/{m_L^i}).
    \end{align*}
    By a cluster we mean a set of the form $\textbf{s}=p_i^{-1}(p_i(y))$ for some $y\in B(L)$, where $\vert \textbf{s}\vert\geq 2.$
    The \textbf{depth of a cluster} $\textbf{s}$ is given by
    \begin{align}\label{Equation: Depth of a cluster}
        d_\textbf{s}:=\frac{max\{i\in \mathbb{Z}_{\geq 0}\vert \textbf{s}=p_i^{-1}(p_i(y)),\mbox{\quad for some $y\in B(L)$}\}}{e}.
    \end{align}
    We call a cluster of size $\vert\textbf{s}\vert \geq 3$ a \textbf{principal cluster}. A cluster of size two is called a \textbf{twin cluster}. We call a cluster an \textbf{odd cluster} (resp. \textbf{even cluster}), if the number of elements in $\textbf{s}$ is odd (resp. even). The \textbf{parent} of a cluster $\textbf{s}$, denoted by $P(\textbf{s})$, is the smallest cluster distinct from $\textbf{s}$ containing it. In this case we say that $\textbf{s}$ is a \textbf{child} of $P(\textbf{s})$. A cluster $\textbf{s}$ is called \"ubereven, if all of its children are even. We always assume that all clusters in $B(L)$ are contained in a cluster of depth zero. We denote this cluster by $\textbf{s}_0$. A proper cluster $\textbf{s}$ is called maximal, if $P(\textbf{s})=\textbf{s}_0.$ Given an arbitrary cluster $\textbf{s}\neq \textbf{s}_0$, we denote the maximal cluster containing $\textbf{s}$ by $\textbf{s}_{max}.$\par
    \begin{definition}\label{Definition:invariant v}
        We denote the unique irreducible component of $\mathcal{Y}^{min}$ by $C_0$ and we define $v_{\varphi}:=v_\varphi(C_0)$ as defined in Definition \ref{Definition:Order-of-vanishing}. For any cluster $\textbf{s}$ in $\Sigma_{X/Y}$, define
        \begin{align}
            v_\textbf{s}=v_\varphi + \sum_{r\in \textbf{s}_0} d_{r\wedge \textbf{s}},
        \end{align}
        where $r\wedge b$ is the smallest cluster in $\Sigma_{X/Y}$ containing $b$ and $r.$
    \end{definition}
    The collection of clusters gives a nested collection of subsets of $B(L)$. The cluster picture of the covering $\varphi\colon X\rightarrow Y$, denoted by $\Sigma_{X/Y},$ is the data of the collection of clusters together with the numerical invariants $d_\textbf{s}$ and $v_{\textbf{s}}$ associated to each cluster. For convenience, when it is clear from the context, we ignore the numerical invariants $d_\textbf{s}$ and $v_{\textbf{s}}$, and we view $\Sigma_{X/Y}$ as a set whose elements are among subsets of $B(L).$
    \begin{remark}\label{Remark:cluster-picture-for-arbitrary-proper-model.}
        We can define the notion of the cluster picture of the covering $\varphi\colon X\rightarrow Y$ with respect to an arbitrary regular proper model of $Y$. In fact, in proving the semi-stability criterion of Theorem \ref{Theorem: Semi-stability-criterion-for 2:1-covers} we need to consider the cluster picture of the covering $\varphi$, with respect to some regular models of $Y$ dominating $\mathcal{Y}^{min}.$ These regular models are constructed inductively using the cluster picture $\Sigma_{X/Y}.$
    \end{remark}
    \begin{example}\label{Example: Difference between the notion of cluster pictures}
        Let $p$ be a prime distinct from 2. Consider the hyperelliptic curve given by the Weierstrass equation:
        \begin{align*}
            C/\mathbb{Q}_p: \quad y^2=x(x-1)(x-1-p^2)(x-1+p^2)(x-p)(x-p^3)(x+p^3).
        \end{align*}
        The cluster picture associated to this hyperelliptic curve in the sense of \cite{DDMM} is given in figure (A), and the cluster picture associated to the covering $C\rightarrow \mathbb{P}^1_F$ is given in figure (B). The distinction between these two figures appear by adding the data of the poles at infinity to the cluster picture. 
        \begin{figure}[h]
            \centering
            \begin{subfigure}[h]{0.2\textwidth}
                \includegraphics[width=\linewidth]{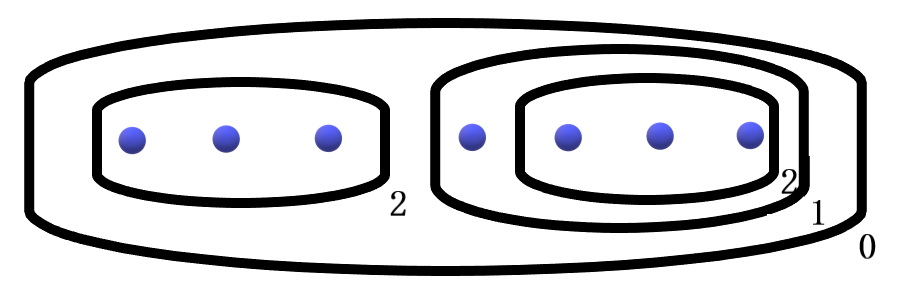}
                \caption{}
            \end{subfigure}
            \hfill
            \begin{subfigure}[h]{0.2\textwidth}
                \centering
                \includegraphics[width=\linewidth]{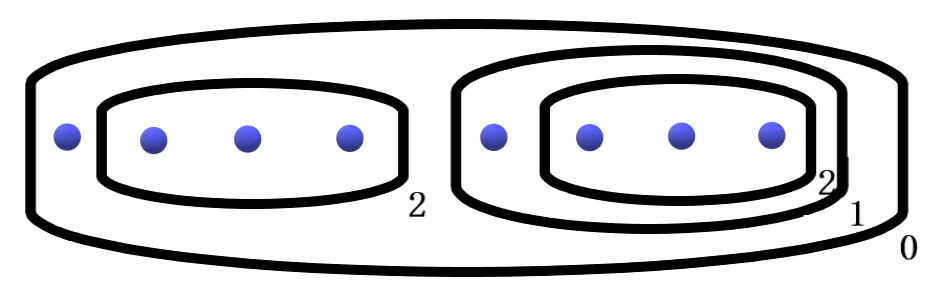}
                \caption{}
                \end{subfigure}
                \caption{Comparing different notion of cluster pictures for hyperelliptic curves.}
                \label{fig:coffee}
            \end{figure}
    \end{example}
    \begin{remark}\label{Remark:Type(A)-Type(B) clusters}
        From now on, whenever we mention the cluster picture associated to a Galois covering, we mean the cluster picture in the sense that is introduced in these notes. However, whenever we talk about a cluster picture associated to a hyperelliptic curve, we mean the cluster picture in the sense of \cite{DDMM}.
    \end{remark}
    \subsection{Semi-stability criterion}\label{Subsection:semi-stabilty-criterion}
    The goal of the remainder of this section, to give a proof of Theorem \ref{Theorem: Semi-stability-criterion-for 2:1-covers}.
    \begin{theorem}\label{Theorem: Semi-stability-criterion-for 2:1-covers}
        The curve $X$ has semi-stable reduction if and only if
        \begin{enumerate}
            \item $v_\varphi$ is even.
            \item $F(B)/F$ is of ramification index at most 2.
            \item Each principal cluster is inertia invariant.
            \item For each principal cluster $\textbf{s}$, $v_\textbf{s}\in 2\mathbb{Z}$ and $d_\textbf{s}\in \mathbb{Z}$.
        \end{enumerate}
    \end{theorem}
    \begin{remark}
        Given a Galois cover $\varphi\colon X\rightarrow Y$, by considering the branch locus $B\subset Y$ of $\varphi$, using \cite[Theorem 3.9]{LL99} we obtain an extension $L/F$ where $X$ is semi-stable over $L.$ This extension is constructed in two steps. First, we consider the base change of $\varphi$ to a cover over the extension $F(B)/F.$ The next step is to construct a good model (\cite[Section 1.8]{LL99}) $\mathcal{Y}$ of $Y$ over $F(B)$, with the property that points in $B_{F(B)}$ specializes to distinct smooth points in $\mathcal{Y}_s$. Now the extension $L/F(B)$ is constructed by 'killing' certain vertical ramifications in $\varphi\colon N(\mathcal{Y},K(X_{F(B)}))
        \rightarrow \mathcal{Y}.$
        However, this extension is not minimal. This can be observed by considering a semi-stable hyperelliptic curve, such that its cluster picture contains a twin cluster of half integer depth. In the case that $\varphi$ is degree 2 and $Y$ admits a smooth proper model, Theorem \ref{Theorem: Semi-stability-criterion-for 2:1-covers} determines a minimal field extension $L/F$ with the property that $X_L$ has semi-stable reduction.
    \end{remark}
    Note that Definition \ref{Definition:invariant v} gives that $v_{\textbf{s}_0}=v_\varphi$.
    Therefore, Theorem \ref{Theorem: Semi-stability-criterion-for 2:1-covers} is a generalization of the semi-stability criterion for hyperelliptic curves given in \cite[Definition 1.7]{DDMM} to the case where $Y$ is not a rational curve.\par
    Using Definition
    \ref{Definition:Reduction-map}, we view the maximal clusters in $\Sigma_{X/Y}$ as points in the special fiber of $\mathcal{Y}^{min}_s.$ These points corresponds to the points in $\overline{\{B\}}\cap \mathcal{Y}^{min}_s.$
    \begin{definition}\cite[Definition 10.1.31]{Liu-Alg-Geom}\label{Definition:Reduction-map}
        Let $Y^{\circ}$ denote the set of closed points of $Y.$ The reduction map of $Y$ with respect to a proper model $\mathcal{Y}$ is:
        \begin{align*}
            r_\mathcal{Y}\colon &Y^{\circ}\rightarrow \mathcal{Y}_s\\
            &y\mapsto \overline{\{y\}} \cap \mathcal{Y}_s,
        \end{align*}
        where $\overline{\{y\}}$ is the closure of this point in $\mathcal{Y}.$
    \end{definition}
    \begin{definition}
        Let $\mathcal{Y}$ be a proper model of $Y$. We call a point $y\in \mathcal{Y}_s$, a critical point of the covering $\varphi\colon X\rightarrow Y$, if $r_\mathcal{Y}^{-1}(y)\cap B(L)$ is non-empty. We denote the set of critical points of the cover $\varphi\colon X\rightarrow Y$ with respect to $\mathcal{Y}^{min}$ by $Crit_{X/Y}.$
    \end{definition}
    \begin{remark}
        The definition of the reduction map $r_\mathcal{Y}$ is dependent on the choice of a proper model of $Y$. Consequently, the critical points of the covering $\varphi\colon X\rightarrow Y$ is also dependent on the choice of a proper model of $Y$. However, when it is clear from the context, we talk about the critical points of a covering and leaving the choice of a model implicit.
    \end{remark}
    \begin{remark}\label{Remark:critical points- as maximal clusters}
        The critical points of the covering $\varphi\colon X\rightarrow Y$ with respect to $\mathcal{Y}^{min}$, corresponds to a maximal proper sub-clusters in $\Sigma_{X/Y}.$ Note that a cluster in $B(L)$ is maximal if it can be written as $r_{\mathcal{Y}}^{-1}(y)\cap B(L)$ for some critical point $y\in \mathcal{Y}_s$. In this sense we may view maximal clusters in $\Sigma_{X/Y}$ as points in $\mathcal{Y}^{min}.$
    \end{remark}
    \begin{definition}\label{Definition:refinable-criticalpoint}
        Let $\mathcal{Y}$ be a proper, semi-stable model of $Y$ and let $y\in \mathcal{Y}_s$ be a critical point of the covering $\varphi\colon X\rightarrow Y.$ The point $y$ is \textbf{refinable}, if there exists a birational map $\pi\colon \mathcal{Y}^{\prime}\rightarrow\mathcal{Y}$ with exceptional locus $E$ satisfying:
        \begin{enumerate}
            \item $\pi(E)=y.$
            \item The blowup             $\pi\colon \mathcal{Y}^{\prime}\rightarrow \mathcal{Y}$ decomposes as a sequence of consecutive blowups centered at closed points.
            \item Denote $\mathcal{X}^{\prime}=N(\mathcal{Y},K(X))$ and let $\rho:\mathcal{X}^{\prime}\rightarrow \mathcal{X}$ be the birational map obtained by contracting every exceptional curves in $\mathcal{X}^{\prime}$. Then, points in $\mathcal{X}\cap \rho(\varphi^{-1}(E))$ are all regular points of $\mathcal{X}$ and $\mathcal{X}_s\cap \rho(\varphi^{-1}(E))$ has at worst nodal singularities. 
            \end{enumerate}
    \end{definition}
    \begin{proposition}\label{Proposition:Semi-stable-if-and-only-if-global-and-local-condition}
        The curve $X$ is semi-stable if and only if $v_\varphi$ is even and all critical points of the covering $\varphi\colon X\rightarrow Y$ with respect to $\mathcal{Y}^{min}$ are refinable.
    \end{proposition}
    \begin{proof}
        Assume that $X$ has semi-stable reduction. As $\varphi$ is a degree 2 Galois cover, by \cite[Theorem 7.3]{LL99} there exists a regular model of $\mathcal{Y}$ of $Y$, where its normalization $\mathcal{X}:=N(\mathcal{Y},K(X))$ is a regular model of $X$. Since $g(Y)>0$, there exists a birational map $\pi\colon\mathcal{Y}\rightarrow \mathcal{Y}^{min}$, which decomposes into a sequence of consecutive blowups centered at some closed points. As $X$ is semi-stable and $\mathcal{X}$ is regular, after contracting all exceptional divisors in $X$ we get a semi-stable model of $X$. In particular, all critical points of this covering are refinable. Note that as $\mathcal{Y}$ dominates $\mathcal{Y}^{min}$, we can view $C_0=(\mathcal{Y}^{min})_s$ as an irreducible component of $\mathcal{Y}_{min}$. As the irreducible components in $\varphi^{-1}(C_0)\subset \mathcal{X}_s$ are with positive genus, by \cite[Theorem 9.3.8]{Liu-Alg-Geom} we must have that these components are reduced, which by Lemma \ref{Lemma:Trivialization-of-the-covering-using-vanishing-order} we get that $v_\varphi$ is even. \par Now assume that all the critical points of the covering $\varphi:X\rightarrow Y$ are refinable and $v_\varphi$ is even. Applying a sequence of blowups satisfying the conditions given in Definition \ref{Definition:refinable-criticalpoint}, starting from each critical point, we construct a model $\mathcal{Y}$. Denote $\mathcal{X}=N(\mathcal{Y},K(X))$. By construction, we get that $\mathcal{X}$ is semi-stable after contracting all of its exceptional divisors. 
    \end{proof}
    \begin{remark}\label{Remark:Classification-of-refinable-critical-points-for-hyperelliptic-curves}
        Proposition \ref{Proposition:Semi-stable-if-and-only-if-global-and-local-condition} is not true under the assumption that $g(Y)=0$. As it is outlined in \cite{DDMM}, it is possible to construct a semi-stable model $\mathcal{Y}$ of $Y$ where $\mathcal{X}=N(\mathcal{Y},K(X))$ is a regular model of $X.$ In this case $\mathcal{Y}$ is constructed from $\mathcal{Y}^{min}$ be a sequence of consecutive blowups centered at smooth closed points on the special fiber. Similar to the proof of Proposition \ref{Proposition:Semi-stable-if-and-only-if-global-and-local-condition}, we can view $(\mathcal{Y}^{min})_s$ as an irreducible component of $\mathcal{Y}_s$ denoted by $C_0$. As $g(C_0)=0$, the irreducible components of $\mathcal{X}_s$ contained in $\varphi^{-1}(C_0)$ can be of arithmetic genus zero, and in particular, they can be exceptional in $\mathcal{X}_s$. This gives rise to cases where $v_{\varphi}$ is not necessarily even and yet $X$ is semi-stable.
    \end{remark}
    \begin{remark}
        The statement of Proposition \ref{Proposition:Semi-stable-if-and-only-if-global-and-local-condition} remains true for Galois cover of curves $\varphi\colon X\rightarrow Y$, such that $\varphi$ has simultaneous resolution of singularities \cite[Section 6]{LL99}. The cover $\varphi$ has simultaneous resolution of singularities, if given an arbitrary finite morphism of normal models $\varphi\colon \mathcal{X}\rightarrow \mathcal{Y}$, there exists a regular model of $Y$ dominating $\mathcal{Y}$, where its normalization is a regular model of $X$. In particular, for the case of degree 3 Galois cover of smooth curves, a similar statement to Proposition \ref{Proposition:Semi-stable-if-and-only-if-global-and-local-condition} remains true.
    \end{remark}
    A morphism of normal models $\varphi\colon \mathcal{X}\rightarrow \mathcal{Y}$, induces a morphism on the formal completion of stalks at any closed point $y\in \mathcal{Y}$,   
    \begin{align*}
        \hat{\varphi}_y\colon \hat{\mathcal{O}}_{\mathcal{Y},y}\rightarrow \hat{\mathcal{O}}_{\mathcal{Y},y}\otimes \mathcal{O}_{\mathcal{X}}.
    \end{align*}
    By Lemma \ref{Lemma:2:1-covers-are-spectral}, there exists a line bundle $\mathcal{L}\in Pic(\mathcal{Y})$ and a section $s\in \mathcal{L}^{\otimes2}(\mathcal{Y})$, such that 
    \begin{align*}
        \mathcal{X}=N(\mathcal{Y},K(X))=S(\mathcal{Y},\mathcal{L},s).
    \end{align*}
    Assume that $y\in \mathcal{Y}_s$ is a smooth closed point. Let $Spec(A)\subset \mathcal{Y}$ be an affine neighborhood of $y\in Y$, such that the line bundle $\mathcal{L}$ trivializes. Let $m\subset A$ be the maximal ideal corresponding to the point $y$. Then,
    \begin{align*}
        \hat{\mathcal{O}}_{\mathcal{Y},y}\otimes _{\mathcal{O}_\mathcal{Y}} \mathcal{\mathcal{O}_\mathcal{X}}&\cong A[x]/(z^2-s)\otimes_A (\lim_{\substack{\longleftarrow}}A/m^n)\\
        &\cong \lim_{\substack{\longleftarrow}}(A[z]/(z^2-s)\otimes A/m^n)\\
        &\cong (\lim_{\substack{\longleftarrow}}A/m^n)[z]/(z^2-s)\\
        &\cong \mathcal{O}_F[[w]][z]/(z^2-\hat{t}_y)\tageq\label{Equation:Local-Trivialization-isomorphism_0}
    \end{align*}
    For some $\hat{t}_y\in \mathcal{O}_F[[w]]$.
    Assuming that $\hat{t}_y\in \hat{\mathcal{O}}_{\mathcal{Y},y}$ is not a unit, using Equation \ref{Equation:Local-Trivialization-isomorphism_0}, we get that
    \begin{align*}           
        \hat{\mathcal{O}}_{\mathcal{Y},y}\otimes _{\mathcal{O}_\mathcal{Y}} \mathcal{\mathcal{O}_\mathcal{X}}\cong \mathcal{O}_F[[w,z]]/(z^2-\hat{t}_y)\tageq\label{Equation:Local-Trivialization-isomorphism}
     \end{align*}
     \begin{remark}
        In deriving equation \ref{Equation:Local-Trivialization-isomorphism}, we are using the fact that $y\in \mathcal{Y}_s$ is smooth and $k_F$ is algebraically closed. In this case, we have that $\hat{\mathcal{O}}_{Y,y}\cong \mathcal{O}_F[[z]].$ Moreover, note that the definition of local trivialization $\hat{t}_y\in \mathcal{O}_F[[w]]$ is well defined up to     multiplication by a unit in $\mathcal{O}_F[[w]].$
    \end{remark}
    \begin{lemma}\label{Lemma: Units-are-squares}
        Different choices of trivialization at a critical point $y\in \mathcal{Y}$, yields isomorphic rings in Equation \ref{Equation:Local-Trivialization-isomorphism}. In other words, given that $f(w)\in{\mathcal{O}}_{F}[[w]]$ is a unit, we get an isomorphism
        \begin{align*}
            \mathcal{O}_F[[w,z]]/(z^2-\hat{t}_y)\cong \mathcal{O}_F[[w,z]]/(z^2-\hat{t}_yf(w)).  
        \end{align*}
    \end{lemma}
    \begin{proof}
        The units of $\mathcal{O}_F[[w]]$ are squares in this ring. To prove this,  We consider two cases. Assume that $u\in \mathcal{O}_F^{\times}$. By applying Hensel's Lemma and considering the fact $k_F$ is an algebraically closed field with characterisitic not equal to 2, we get that $u$ is a square in $\mathcal{O}_F.$ Now let the unit $u\in \mathcal{O}_F[[w]]^*$ be of the form $u=1+wg$, for some $g\in \mathcal{O}_F[[w]]\setminus \{0\}.$ Since 2 is invertible in $\mathcal{O}_F[[w]]$, we get that $u$ is a square in $\mathcal{O}_F[[w]]$. Any element in $\mathcal{O}_F[[w]]^\times$ can be presented as multiplication of units considered in these two cases. Therefore all units of $\mathcal{O}_F[[w]]$ are squares in this ring. The isomorphism in this lemma is given by the change of variables $w=w$ and $z={z}/{\sqrt{f(w)}}.$
    \end{proof}
    \begin{remark}\label{Remark: Polynomial type of the local trivialization}
        Let $\mathcal{Y}$ be an arbitrary proper regular model of $Y$. Let $y\in \mathcal{Y}_s$ be a non-singular critical point of the covering $\varphi\colon X\rightarrow Y$. Set $r=\vert r_{\mathcal{Y}}^{-1}(y)\cap B(L)\vert$. By a change of variable, if need be, we can assume that
        \begin{align*}
            \hat{t}_y= w^r - \pi_F g(w),
        \end{align*}
        where $g(w)\in \mathcal{O}_F[[w]]$ is a polynomial of degree at most $r-1.$
    \end{remark}
    In the remaining part of this section we are going to show that the property of being a refinable critical point is local in the sense of Lemma \ref{Lemma:Locally-analytic-structure-of-critical-points}.
    \begin{lemma}\label{Lemma:Locally-analytic-structure-of-critical-points}
        For $i=1,2$, let $\varphi_i\colon X_i\rightarrow Y_i$ be a degree two Galois cover of smooth curves and let $\mathcal{Y}_i$ be a regular model of $Y_i$. Define $\mathcal{X}_i=N(\mathcal{Y}_i,K(X_i))$. Let $y_i\in \mathcal({Y}_i)_s$ be a smooth critical point of the covering $\varphi_i\colon X_i\rightarrow Y_i$ with respect to the model $\mathcal{Y}_i$. Assume that there exists a vertical isomorphism making the following diagram commute:
        \begin{equation}\label{Diagram:Same-trivialization-square}
            \begin{tikzcd}
                {\hat{\mathcal{O}}_{\mathcal{Y}_1,y_1}} \arrow[r, "(\hat{\varphi}_1)_{y_1}"] \arrow[d, "\cong"] & {\hat{\mathcal{O}}_{\mathcal{Y}_1,y_1}\otimes_{\mathcal{O}_{\mathcal{Y}_1,y_1}}\mathcal{O}_{\mathcal{X}_1}} \arrow[d, "\cong"] \\
                    {\hat{\mathcal{O}}_{\mathcal{Y}_2,y_2}} \arrow[r, "(\hat{\varphi}_2)_{y_2}"]                    & {\hat{\mathcal{O}}_{\mathcal{Y}_2,y_2}\otimes_{\mathcal{O}_{\mathcal{Y}_2,y_2}}\mathcal{O}_{\mathcal{X}_2}}.   
            \end{tikzcd}
        \end{equation}
        Then,
        \begin{enumerate}
            \item $y_1$ is a refinable critical point for $\varphi_1$ if and only if $y_2$ is a refinable critical point for $\varphi_2.$
            \item Given that $y_1$ is refinable and $\pi_1\colon \mathcal{Y}^{\prime}_1\rightarrow \mathcal{Y}_1$ is a birational map of regular models with exceptional locus $E_1$ satisfying the conditions of Definition \ref{Definition:refinable-criticalpoint}, there exists a birational map of regular models $\pi_2\colon \mathcal{Y}^{\prime}_2\rightarrow \mathcal{Y}_2$ with exceptional locus $E_2$ satisfying the conditions of Definition \ref{Definition:refinable-criticalpoint} such that $\varphi^{-1}_1(E_1)\cap N(\mathcal{Y}^{\prime}_1,K(X_1))_s$ is isomorphic to $\varphi^{-1}_2(E_2)\cap N(\mathcal{Y}^{\prime}_2,K(X_2))_s$.
        \end{enumerate}
        In this case we say that $\varphi_1$ has the same trivialization at $y_1\in \mathcal{Y}_1$ as $\varphi_2$ at $y_2\in\mathcal{Y}_2$.
    \end{lemma}
    \begin{proof}
        The commutative Diagram \ref{Diagram:Same-trivialization-square} implies the existence of an isomorphism
        \begin{align*}
            \rho\colon\hat{\mathcal{O}}_{\mathcal{Y}_1,y_1}\rightarrow \hat{\mathcal{O}}_{\mathcal{Y}_2,y_2},
        \end{align*}
        mapping $\hat{t}_{y_1}$ to $\hat{t}_{y_2}.$
        Let $\pi_i\colon\mathcal{Y}_i^{\prime}\rightarrow \mathcal{Y}_i$ be the blowup centered at $y_i$ with exceptional locus $E_i$ and let $\mathcal{X}^{\prime}_i=N(\mathcal{Y}^{\prime}_i,K(X_i))$ for $i=1,2$. As blowing up at a closed point is local, the morphism $\rho\colon \hat{\mathcal{O}}_{\mathcal{Y}_1,y_1}\rightarrow \hat{\mathcal{O}}_{\mathcal{Y}_2,y_2}$ induces a bijection between closed points in $E_1\cap \mathcal{Y}^{\prime}_1$ and $E_2\cap \mathcal{Y}^{\prime}_2$, also denoted by $\rho$, with the property that for any closed point $y^{\prime}_1\in E_1\cap \mathcal{Y}^{\prime}_1$ we have \[\hat{\mathcal{O}}_{\mathcal{Y}^{\prime}_1,y^{\prime}_1}\otimes_{\mathcal{O}_{\mathcal{Y}^{\prime}_1,y^{\prime}_1}}\mathcal{O}_{\mathcal{X}^{\prime}_1}\cong \hat{\mathcal{O}}_{\mathcal{Y}^{\prime}_2,\rho(y^{\prime}_1)}\otimes_{\mathcal{O}_{\mathcal{Y}^{\prime}_2,\rho(y^{\prime}_1)}}\mathcal{O}_{\mathcal{X}^{\prime}_2}.\]
        Denote the irreducible component $E_i\cap \mathcal{Y}^{\prime}_i$ by $C_i$ for $i=1,2$. Then, we note that $v_{\varphi_1}(C_1)$ is equal to $v_{\varphi_2}(C_2)$ (see Definition \ref{Definition:Order-of-vanishing}). In particular, we see that $\varphi_1^{-1}(E_1)\cap(\mathcal{X}^{\prime}_1)_s$ is isomorphic with $\varphi_1^{-1}(E_2)\cap(\mathcal{X}^{\prime}_2)_s$.
    \end{proof}
    \begin{remark}
        In the setting of Lemma \ref{Lemma:Locally-analytic-structure-of-critical-points}, assume that $\hat{t}_{y_i}\in \hat{\mathcal{O}}_F[[w_i]]$ is the trivialization of the covering $\varphi_i$ at $y_i$ for $i=1,2.$ Using the isomorphism given in \ref{Equation:Local-Trivialization-isomorphism}, we can rewrite the commutative diagram \ref{Diagram:Same-trivialization-square} as follow
        \begin{equation*}
            \begin{tikzcd}
                {\mathcal{O}_F[[z]]} \arrow[r, hook] \arrow[rd, hook] & {\mathcal{O}_F[[z,w_1]]/(z^2-\hat{t}_{y_1})} \arrow[d, "\cong"] \\                      &{\mathcal{O}_F[[z,w_2]]/(z^2-\hat{t}_{y_2}}).
            \end{tikzcd}
        \end{equation*}
        Note the data of the commutative diagram is equivalent to the existence of an isomorphism $\hat{\mathcal{O}}_F[[w_1]]\rightarrow \hat{\mathcal{O}}_F[[w_2]]$ mapping $\hat{t}_{y_1}$ to $\hat{t}_{y_2}.$
    \end{remark}
    Using Lemma \ref{Lemma:Critical-point-can-be-viewed-as-points-in-hyperelliptic-curve}, we view the critical points of the cover $\varphi$, as critical points on a hyperelliptic curve. This allows us to relate the arithmetic properties of $X$, to the arithmetic properties of hyperelliptic curves.
    \begin{lemma}\label{Lemma:Critical-point-can-be-viewed-as-points-in-hyperelliptic-curve}
        Let $y\in \mathcal{Y}^{min}_s$ be a critical point of the covering $\varphi\colon X\rightarrow Y$. Then there exists a hyperelliptic curve $\rho_y\colon C\rightarrow \mathbb{P}^1_F$ together with a critical point $y^{\prime}$, such that $\varphi$ has the same trivialization at $y$ as $\rho$ at $y^{'}.$ In particular, $y$ is a refinable critical point for $\varphi$ if and only if $y^{\prime}$ is a refinable critical point for $\rho.$
    \end{lemma}
    \begin{proof}
        Let $\hat{t}_\textbf{y}\in\hat{\mathcal{O}}_{\mathcal{Y}^{min},y}\cong \mathcal{O}_F[[w]]$ be the trivialization of the covering $\varphi$ at $y\in \mathcal{Y}^{min}$. Let $f(w)$ be a monic polynomial of degree 3, with no roots in common with $\hat{t}_y$ modulo $m_F.$ Define $\rho\colon C\rightarrow \mathbb{P}_F^1$ be the hyperelliptic curve given by the equation \begin{align*}
        C: \quad y^2=f(w)\hat{t}_{\textbf{s}}.
        \end{align*}The extra factor in the definition of $C$, guarantees that it is a hyperelliptic curve. The covering $\rho\colon C\rightarrow \mathbb{P}^1_F$ has a critical point $y^{\prime}$ with respect to $\mathbb{P}^1_{\mathcal{O}_F}$ such that
        \[\hat{\mathcal{O}}_{\mathbb{P}^1_{\mathcal{O}_F},\textbf{s}^\prime}
        \otimes_{\mathcal{O}_{\mathbb{P}^1_{\mathcal{O}_F}}}\mathcal{C}\cong \mathcal{O}_F[[z,w]]/(z^2-\hat{t}_\textbf{s}),\]
        where $\mathcal{C}=N(\mathbb{P}_{\mathcal{O}_F}^1,K(C)).$ Now we can apply Lemma \ref{Lemma:Critical-point-can-be-viewed-as-points-in-hyperelliptic-curve} to the critical points $y\in \mathcal{Y}^{min}_s$ and $y^{\prime}\in \mathbb{P}^{1}_{\mathcal{O}_F}.$ 
    \end{proof}
    \begin{definition}\label{Definition: hyperelliptic associated to maximal subclusters}
        Let $\textbf{s}\in \Sigma_{X/Y}$ be a maximal cluster corresponding to a critical point $y$ in $\mathcal{Y}^{min}_s$. We denote the corresponding hyperelliptic curve given in Lemma \ref{Lemma:Critical-point-can-be-viewed-as-points-in-hyperelliptic-curve} by $C_\textbf{s}$ or $C_{y}$. 
    \end{definition}
    \subsection{Proof of theorem \ref{Theorem: Semi-stability-criterion-for 2:1-covers}}
    By Proposition \ref{Proposition:Semi-stable-if-and-only-if-global-and-local-condition}, $X$ is semi-stable if and only if $v_\varphi$ is even and all of the critical points of the covering $\varphi\colon X\rightarrow Y$ with respect to the model $\mathcal{Y}^{min}$ are refinable. By Lemma \ref{Lemma:Critical-point-can-be-viewed-as-points-in-hyperelliptic-curve}, we can view critical points of the covering $\varphi\colon X\rightarrow Y$ as critical points arising for some hyperelliptic curves. Now the conditions (2)-(4) given in Theorem \ref{Theorem: Semi-stability-criterion-for 2:1-covers} follows from the semi-stability criterion for hyperelliptic curve as given in \cite[Definition 1.7]{DDMM}.
    \subsection{A description of $\mathcal{X}^{min}_s$  assuming that $X$ has semi-stable reduction}
     Throughout this subsection we assume that $\varphi\colon X\rightarrow Y$ satisfies the conditions given in Assumption \ref{Assumption:Main assumption}. In other words, we assume that $X$ has semi-stable reduction over $F$.\par
     Following the construction given in \cite{DDMM}, we construct a semi-stable model  $\mathcal{Y}^{disc}/\mathcal{O}_F$ of $Y$, where its normalization in $K(X)$ is a regular model of $X.$ Using the results given in Subsection \ref{Subsection: Model of curves}, we study the irreducible components of $\mathcal{X}_s$, after identifying all of the exceptional divisors and contracting these components, we get a semi-stable model of $X$. To describe $\mathcal{X}^{min}_s$ we adopt and use the notion of admissible collection of discs as given \cite{DDMM}.
    \begin{definition}
        Let $A$ be a finite set. An abstract disc on $A$ is the data of a a subset $D\subset A$ together with a rational invariant $d_D$ assigned to it, called the \textbf{depth} of $D.$ A disc is \textbf{integral}, if it has integer depth. We say that $D$ is a subdisc of $D^{\prime}$, if $D\subset D^{\prime}$ and $d_D>d_{D^{\prime}}$. We denote this relation by $D< D^{\prime}$. An \textbf{admissible collection of discs on $A$} is a finite collection of integral discs $\mathcal{D}=\{D_i\}_{i\in I}$ on $A$ satisfying
        \begin{enumerate}
            \item $D_i=A$ for some $i\in A$.
            \item for any $i\neq j$ if $D_i\cap D_j\neq \emptyset$ then either $D_i\leq D_j$ or $D_j\leq D_i.$ In other words, the collection of discs in $\mathcal{D}$, gives a nested collection of subsets of $A$.
            \item The collection of discs in $\mathcal{D}$ is \textbf{complete}. That is to say given that $D_1 < D_2$ in $\mathcal{D}$, where $d_{D_1}>d_{D_2}$, there exists a disc $D\in \mathcal{D}$ such that $D_1<D<D_2.$ 
        \end{enumerate}
    \end{definition}
    \begin{definition}\label{Definition: Isomorphism of admissible collection}
        For $i\in\{1,2\}$, let $\mathcal{D}_i$ be a collection of admissible discs on the underlying set $A_i$. An isomorphism of admissible discs between $\mathcal{D}_1$ and $\mathcal{D}_2$, is a bijection $f\colon A_1\rightarrow A_2$ satisfying the following conditions.
        \begin{enumerate}
        \item Given $D_1\in \mathcal{D}_1$, then $f(D_1)\in \mathcal{D}_2$ and $d_{D}=d_{f(D)}.$
        \item $f$ induces a bijection from discs in $\mathcal{D}_1$ to discs on $\mathcal{D}_2.$
    \end{enumerate}
    We denote the isomorphism of discs by $f\colon \mathcal{D}_1\xrightarrow{\cong} \mathcal{D}_2.$ The group of automorphisms of collection of admissible discs on $D_1$ is denoted by $Aut(\mathcal{D}_1)$. We say that a group $G$ acts on $\mathcal{D}_1$ if there exists a homomorphism $G\rightarrow Aut(\mathcal{D}_1).$
    \end{definition}
    Since $X$ is semi-stable, all clusters on $\Sigma_{X/Y}$ are with integral depth, possibly with the exception of some twin clusters. This fact follows from \cite[Proposition C.7]{DDMM}. In particular, a cluster with non-integer depth has no proper sub cluster.
    \begin{definition}
        The collection of admissible disc $\mathcal{D}_{X/Y}$ on $B(L)$ is defined as follow.
        \begin{enumerate}
            \item To each cluster $\textbf{s}\in \Sigma_{X/Y}$ with integral depth, we assign a disc denoted by $D(\textbf{s})$, equal to $\textbf{s}$ as a set and with depth $d_\textbf{s}.$ The disc $D(\textbf{s})$ is called a \textbf{ defining disc} in $\mathcal{D}_{X/Y}$.
            \item Given a cluster $\textbf{s}$ with non-integer depth $d_\textbf{s}>1$, define $D(\textbf{s})$ to be the disc with depth $\lfloor d_\textbf{s}\rfloor$ and equal to $\textbf{s}$ as a set.
            \item Given a cluster $\textbf{s}\neq \textbf{s}_0,$ 
            the collection $\mathcal{D}_{X/Y}$ contains all integral discs $D \subset B(L)$ satisfying $D(P(\textbf{s})<D<D(\textbf{s})$, such that $d_{P(\textbf{s})}<d_D<d_{\textbf{s}}$
        \end{enumerate}
        We set $D_0$ to be the disc of depth zero containing $B(L).$ Given a disc $D\subset \mathcal{D}_{X/Y}$ where $D\neq D_0$, set $P(D)$ to be the  be the unique disc in $\mathcal{D}_{X/Y}$ such that \[D<P(D) \mbox{ and } d_{p(D)}=d_{D}-1.\] When it is clear from the context, we drop the subscript and denote the collection of admissible disc associated to the cover $\varphi\colon X\rightarrow Y$ by $\mathcal{D}.$
    \end{definition}
     Note that $\mathcal{D}_{X/Y}$ is an admissible collection of integral discs on $B(L).$
    \begin{definition}\label{Definition: Subcollection associated to a disc}
        Let $y\in \mathcal{Y}^{min}_s$ be a critical point of the covering $\varphi\colon X\rightarrow Y.$ Then, this critical point corresponds to a maximal disc $D\in \mathcal{D}_{X/Y}.$ Set $\mathcal{D}_y\subset \mathcal{X}_{X/Y}$ to be the collection of discs consists of $D$ and all discs in $\mathcal{D}_{X/Y}$ that are contained in $D$.
    \end{definition}
    \begin{remark}
        Let $y\in \mathcal{Y}^{min}$ be a critical point of the covering $\varphi\colon X\rightarrow Y$. Let the pair $\rho_y\colon C_y\rightarrow \mathbb{P}^1_F=:Z$ and $z\in \mathcal{Z}^{min}$ be as they are given in Lemma \ref{Lemma:Critical-point-can-be-viewed-as-points-in-hyperelliptic-curve}. Then, there exists an isomorphism of collection of admissible discs $\mathcal{D}_z\cong \mathcal{D}_y.$ 
    \end{remark}
    \begin{definition}
        A disc $D\in \mathcal{D}_{X/Y}$ is called even (resp. odd), if $\vert D\cap B(L)\vert $ is even (resp. odd). A disc $D$ is called \"ubereven, if it is even and all immediate sub-disc of $D$ are also even.
    \end{definition}
    Using the admissible collection of discs $\mathcal{D}_{X/Y}$, we construct a minimal regular model of $Y$, denoted by $\mathcal{Y}^{disc}$, with the property that its normalization in $K(X)$ is a regular model of $X$. 
    \begin{definition}
        Define $T_{X/Y}$ to be the rooted tree with vertices $v_D$ corresponding to discs $D\in \mathcal{D}_{X/Y}$, with edges ${D_i}{D_j}$ if there is a parent/child relation between $D_i$ and $D_j$. We set the roots of this tree to be the vertex corresponding to the disc $D_0$.
    \end{definition}
    \begin{example}
    Let $\varphi:X\rightarrow Y$ be the hyperelliptic curve given in Example \ref{Example: Difference between the notion of cluster pictures}. Then $T_{X/Y}$ is given by the following graph.
        \begin{center}
            \begin{tikzpicture}[x=0.75pt,y=0.75pt,yscale=-1,xscale=1]

\draw  [fill={rgb, 255:red, 0; green, 0; blue, 0 }  ,fill opacity=1 ] (279.06,27.08) .. controls (279.06,25.38) and (280.09,24) .. (281.36,24) .. controls (282.63,24) and (283.66,25.38) .. (283.66,27.08) .. controls (283.66,28.79) and (282.63,30.17) .. (281.36,30.17) .. controls (280.09,30.17) and (279.06,28.79) .. (279.06,27.08) -- cycle ;
\draw  [fill={rgb, 255:red, 0; green, 0; blue, 0 }  ,fill opacity=1 ] (312.33,60.86) .. controls (312.33,59.16) and (313.36,57.78) .. (314.63,57.78) .. controls (315.9,57.78) and (316.93,59.16) .. (316.93,60.86) .. controls (316.93,62.56) and (315.9,63.95) .. (314.63,63.95) .. controls (313.36,63.95) and (312.33,62.56) .. (312.33,60.86) -- cycle ;
\draw  [fill={rgb, 255:red, 0; green, 0; blue, 0 }  ,fill opacity=1 ] (279.06,60.22) .. controls (279.06,58.52) and (280.09,57.14) .. (281.36,57.14) .. controls (282.63,57.14) and (283.66,58.52) .. (283.66,60.22) .. controls (283.66,61.93) and (282.63,63.31) .. (281.36,63.31) .. controls (280.09,63.31) and (279.06,61.93) .. (279.06,60.22) -- cycle ;
\draw  [fill={rgb, 255:red, 0; green, 0; blue, 0 }  ,fill opacity=1 ] (245.32,60.86) .. controls (245.32,59.16) and (246.35,57.78) .. (247.62,57.78) .. controls (248.89,57.78) and (249.92,59.16) .. (249.92,60.86) .. controls (249.92,62.56) and (248.89,63.95) .. (247.62,63.95) .. controls (246.35,63.95) and (245.32,62.56) .. (245.32,60.86) -- cycle ;
\draw    (281.36,27.08) -- (314.63,60.86) ;
\draw    (281.36,27.08) -- (247.62,60.86) ;
\draw    (281.36,27.08) -- (281.36,60.22) ;
\draw    (281.36,60.22) -- (281.29,90.08) ;
\draw  [fill={rgb, 255:red, 0; green, 0; blue, 0 }  ,fill opacity=1 ] (278.99,90.08) .. controls (278.99,88.37) and (280.02,86.99) .. (281.29,86.99) .. controls (282.56,86.99) and (283.59,88.37) .. (283.59,90.08) .. controls (283.59,91.78) and (282.56,93.16) .. (281.29,93.16) .. controls (280.02,93.16) and (278.99,91.78) .. (278.99,90.08) -- cycle ;
\draw  [fill={rgb, 255:red, 0; green, 0; blue, 0 }  ,fill opacity=1 ] (312.33,90.18) .. controls (312.33,88.48) and (313.36,87.09) .. (314.63,87.09) .. controls (315.9,87.09) and (316.93,88.48) .. (316.93,90.18) .. controls (316.93,91.88) and (315.9,93.26) .. (314.63,93.26) .. controls (313.36,93.26) and (312.33,91.88) .. (312.33,90.18) -- cycle ;
\draw    (314.63,60.86) -- (314.63,90.18) ;
\draw  [fill={rgb, 255:red, 0; green, 0; blue, 0 }  ,fill opacity=1 ] (341.79,91.45) .. controls (341.79,89.75) and (342.82,88.37) .. (344.09,88.37) .. controls (345.36,88.37) and (346.39,89.75) .. (346.39,91.45) .. controls (346.39,93.16) and (345.36,94.54) .. (344.09,94.54) .. controls (342.82,94.54) and (341.79,93.16) .. (341.79,91.45) -- cycle ;
\draw    (314.63,60.86) -- (344.09,91.45) ;
\draw  [fill={rgb, 255:red, 0; green, 0; blue, 0 }  ,fill opacity=1 ] (341.79,123.96) .. controls (341.79,122.25) and (342.82,120.87) .. (344.09,120.87) .. controls (345.36,120.87) and (346.39,122.25) .. (346.39,123.96) .. controls (346.39,125.66) and (345.36,127.04) .. (344.09,127.04) .. controls (342.82,127.04) and (341.79,125.66) .. (341.79,123.96) -- cycle ;
\draw  [fill={rgb, 255:red, 0; green, 0; blue, 0 }  ,fill opacity=1 ] (318.98,124.59) .. controls (318.98,122.89) and (320.01,121.51) .. (321.28,121.51) .. controls (322.55,121.51) and (323.58,122.89) .. (323.58,124.59) .. controls (323.58,126.3) and (322.55,127.68) .. (321.28,127.68) .. controls (320.01,127.68) and (318.98,126.3) .. (318.98,124.59) -- cycle ;
\draw    (344.09,91.45) -- (321.28,124.59) ;
\draw    (344.09,91.45) -- (344.09,123.96) ;
\draw  [fill={rgb, 255:red, 0; green, 0; blue, 0 }  ,fill opacity=1 ] (368.4,123.32) .. controls (368.4,121.62) and (369.43,120.24) .. (370.7,120.24) .. controls (371.97,120.24) and (373,121.62) .. (373,123.32) .. controls (373,125.02) and (371.97,126.4) .. (370.7,126.4) .. controls (369.43,126.4) and (368.4,125.02) .. (368.4,123.32) -- cycle ;
\draw    (344.09,91.45) -- (370.7,123.32) ;
\draw  [fill={rgb, 255:red, 0; green, 0; blue, 0 }  ,fill opacity=1 ] (304.25,124.59) .. controls (304.25,122.89) and (305.28,121.51) .. (306.55,121.51) .. controls (307.82,121.51) and (308.85,122.89) .. (308.85,124.59) .. controls (308.85,126.3) and (307.82,127.68) .. (306.55,127.68) .. controls (305.28,127.68) and (304.25,126.3) .. (304.25,124.59) -- cycle ;
\draw    (281.29,90.08) -- (306.55,124.59) ;
\draw  [fill={rgb, 255:red, 0; green, 0; blue, 0 }  ,fill opacity=1 ] (279.54,123.32) .. controls (279.54,121.62) and (280.57,120.24) .. (281.84,120.24) .. controls (283.11,120.24) and (284.14,121.62) .. (284.14,123.32) .. controls (284.14,125.02) and (283.11,126.4) .. (281.84,126.4) .. controls (280.57,126.4) and (279.54,125.02) .. (279.54,123.32) -- cycle ;
\draw    (281.29,90.08) -- (281.84,123.32) ;
\draw  [fill={rgb, 255:red, 0; green, 0; blue, 0 }  ,fill opacity=1 ] (253.4,122.68) .. controls (253.4,120.98) and (254.43,119.6) .. (255.7,119.6) .. controls (256.97,119.6) and (258,120.98) .. (258,122.68) .. controls (258,124.38) and (256.97,125.77) .. (255.7,125.77) .. controls (254.43,125.77) and (253.4,124.38) .. (253.4,122.68) -- cycle ;
\draw    (281.29,90.08) -- (255.7,122.68) ;

\end{tikzpicture}
    \end{center}
    \end{example}
    \begin{remark}\label{Remark: Cluster pictures as a subset of L}
        Let $y\in Crit_{X/Y}$ be a critical point. Then, we can view the cluster picture $\mathcal{D}_y$ as a nested collection of subsets inside $Y_{+}(y)(L)$, where $Y_{+}(y)$ is the formal fiber of $\mathcal{Y}^{min}$ over $y$ (\cite[Definition 10.1.39]{Liu-Alg-Geom}. In particular, since $y$ is a smooth point of $\mathcal{Y}^{min}$, the cluster picture $\mathcal{D}_y$, can be viewed as a collection of nested subsets inside $m_L\subset L$ (\cite[Proposition  10.1.40]{Liu-Alg-Geom}).
    \end{remark}
    \begin{definition}\label{Definition: Center of a cluster}
        Let $y\in \mathcal{Y}$ be a critical point of the cover $\varphi.$ Given a proper cluster $\textbf{s}$ in $\mathcal{D}_y$, the center of this cluster is a point $z_\textbf{s}$, satisfying
        $v_{\overline{F}}(z_\textbf{s}-r)\geq d_\textbf{s}$, for all $r\in \textbf{s}$. Here, $v_{\overline{F}}$ is the valuation of the field $\overline{F}$ extending $v_F$, and $\mathcal{D}_y$ is viewed as a collection of nested open subsets of $L$, as explained in Remark \ref{Remark: Cluster pictures as a subset of L}.
    \end{definition}
    \begin{lemma}\label{Lemma:Existence of branch separating model}
        There exists a semi-stable model $\mathcal{Y}/\mathcal{O}_F$ of $Y$, such that the closure of the branch locus $\mathcal{B}=\overline{\{B\}}\subset \mathcal{Y}$, satisfies the following properties.
        \begin{enumerate}
            \item Points in $\mathcal{B}\cap \mathcal{Y}_s$ consist of smooth points of $\mathcal{Y}_s.$
            \item Horizontal component in  $\mathcal{B}$ do not intersect each other, with the exception of pair of elements contained in a cluster with non-integer depth.
            \end{enumerate}
    \end{lemma}
    \begin{proof}
        If the cluster picture associated to the covering $\varphi\colon X\rightarrow Y$ contains no cluster with non-integer depth, then each horizontal component in $\overline{\{B\}}\subset \mathcal{Y}^{min}$ is smooth. Now the proof of this lemma follows from \cite[Lemma 1.9]{LL99}.\par
        Assume that that the cluster picture of $\varphi\colon X\rightarrow Y$ contains clusters with non-integer depth. Let $\textbf{s}\subset B(L)$ be a twin cluster with non-integer depth. Since the cluster picture $\Sigma_{X/Y}$ satisfies the conditions given in Theorem \ref{Theorem: Semi-stability-criterion-for 2:1-covers}, by applying \cite[Lemma B.1]{DDMM}, we can find a rational point $z_\textbf{s}\in Y(F)$, where ${z}_{\textbf{s}}$ is the center of the cluster $\textbf{s}.$\par
        Construct a divisor $B^{\prime} \subset Y(F)$ by substituting each twin cluster $\textbf{s}$ with non-integer depth in $B(L)$ with a rational center $z_\textbf{s}\in Y(F)$. By applying \cite[Lemma 1.9]{LL99}, we can construct a semi-stable model $\mathcal{Y}$ separating horizontal component of $\overline{\{B^{\prime}\}}.$ By the way that $B^{\prime}$ is constructed, it follows that $\mathcal{Y}$ is a model satisfying the conditions given in this lemma.
    \end{proof}
    Next, we give an explicit construction of a model of $Y$, denoted by $\mathcal{Y}^{disc}$, and we show that it is a regular model of $Y$ with dual graph isomorphic to $T_{X/Y}.$ In particular, evident from Construction \ref{Construction:Y^{disc}}, one can see that $\mathcal{Y}^{disc}$ satisfies the conditions given in Lemma \ref{Lemma:Existence of branch separating model}.
    \begin{construction}\label{Construction:Y^{disc}}
        Starting from $\mathcal{Y}^{min}$, define $\pi_1\colon \mathcal{Y}_1\rightarrow \mathcal{Y}^{min}$ to be the blowup map centered at a critical point $\textbf{s}\subset \mathcal{Y}^{min}$ where $d_\textbf{s}>\frac{1}{2}$. Assuming that $\mathcal{Y}_i$ is constructed, we define $\pi_{i+1}\colon \mathcal{Y}_{i+1}\rightarrow \mathcal{Y}_i$ to be the blowup at a critical point $\textbf{s}\subset \mathcal{Y}_i$ where $\textbf{s}$ is a critical point of the covering $\varphi\colon X\rightarrow Y$ with respect to the proper model $\mathcal{Y}_i$ and $d_{\textbf{s}}>\frac{1}{2}.$ This process terminates as there are finitely many clusters in $\Sigma_{X/Y}.$ We denote the model of $Y$ constructed above by $\mathcal{Y}^{disc}$.
    \end{construction}
    \begin{definition}\label{Definition: pi^{Disc}}
        By Construction \ref{Construction:Y^{disc}}, the model $\mathcal{Y}^{disc}$ is birational to $\mathcal{Y}^{min}$. We denote this birational map by $\pi^{disc}\colon \mathcal{Y}^{disc}\rightarrow\mathcal{Y}^{min}$.
    \end{definition}
    It follows from Construction \ref{Construction:Y^{disc}},  the birational map $\pi^{disc}$ decomposes as a sequence of consecutive blowups
    \begin{align}
        \mathcal{Y}^{disc}\cong\mathcal{Y}_n\xrightarrow{\pi_{n}}\mathcal{Y}_{n-1}\xrightarrow{\pi_{n-1}}\dots \mathcal{Y}_1\xrightarrow{\pi_1}\mathcal{Y}_0=\mathcal{Y}^{min},\tageq\label{Equation:Decompostion of pi^disc}
    \end{align}
    where each birational map $\pi_i$ is a blowup centered at a smooth closed point of $(\mathcal{Y}_ {i-1})_s.$
    \begin{remark}\label{Remark: Disk-Component correspondence}
        We can identify the irreducible components of $\mathcal{Y}^{disc}$ with discs in $\mathcal{D}_{X/Y}.$ Critical points of the covering $\varphi\colon X\rightarrow Y$ with respect to the model $\mathcal{Y}^{min}$ are in bijection with maximal discs in $\mathcal{D}_{X/Y}$. Assume that $\pi_1\colon \mathcal{Y}_1\rightarrow \mathcal{Y}^{min}$ is the blowup centered at the critical point corresponding to the maximal disc $D$. Denote the irreducible component of $\mathcal{Y}$ contained in the exceptional locus of $\pi_1$ by $\Gamma_D.$ Define $\mathcal{D}_{X/Y}^{\mathcal{Y}_1}$ from $\mathcal{D}_{X/Y}$, by deleting $D$ from this collection and subtracting one from the depth of all subdisc of $D$. Now the critical point of the covering $\varphi\colon X\rightarrow Y$ with respect to $\mathcal{Y}_1$ corresponds to maximal discs in $\mathcal{D}_{X/Y}^{\mathcal{Y}_1}.$ Following this process, the birational map $\pi_{i+1}\colon \mathcal{Y}_{i+1}\rightarrow \mathcal{Y}_i$ is given by the blowup centered at a critical point corresponding to a maximal disc in $\mathcal{D}_{X/Y}^{\mathcal{Y}_{i}}$. Given a disc $D\subset \mathcal{D}_{X/Y}$, we denote the corresponding irreducible component to this disc in $\mathcal{Y}^{disc}_s$ by $\Gamma_D.$
    \end{remark}
    \begin{definition}
        For any disc $D\in \mathcal{D}_{X/Y}$ define:
        \begin{align}\label{Equation: vD}
            v_D=v_\varphi+ \Sigma_{r\in D_{max}}d_{r\wedge D}.
        \end{align}
        If $D=D_0$, we set $v_D=v_\varphi.$
    \end{definition}
    \begin{remark}
        For $\textbf{s}\subset B(L)$ a principal cluster, we have $v_\textbf{s}=v_{D(\textbf{s})}.$
    \end{remark} 
    \begin{lemma}\label{Lemma:Vanishing-order-mod-2-on-Y^{disc}}
        The model $\mathcal{Y}^{disc}$ is a semi-stable, regular proper model of $Y$ with dual graph isomorphic $T_{X/Y}.$ Given $D\in \mathcal{D}_{X/Y}$, we have:
        \begin{align}
            v_\varphi({\Gamma_D})=v_D \mbox{ (mod 2)} \tageq\label{Equation: Mod 2 congruency}.
        \end{align}
    \end{lemma}
    \begin{proof}
        The fact that $\mathcal{Y}^{disc}$ is semi-stable follows from the construction of this model. Note that in Construction \ref{Construction:Y^{disc}}, $\mathcal{Y}_{i+1}$ is constructed from $\mathcal{Y}_{i}$ by a blowup centered at a smooth closed point in $(\mathcal{Y}_i)_s$, therefore assuming $\mathcal{Y}_i$ is a regular proper semi-stable model of $Y$, then so is $\mathcal{Y}_{i+1}$. By an inductive argument we see that $\mathcal{Y}^{disc}$ is a proper semi-stable model of $Y$ with the dual graph $T_{X/Y}$.\par
        The birational map $\pi^{disc}\colon \mathcal{Y}^{disc}\rightarrow \mathcal{Y}^{min}$ decomposes as a sequence of consecutive blowups 
        \begin{align*}
            \mathcal{Y}^{disc}\cong\mathcal{Y}_n\xrightarrow{\pi_{n}}\mathcal{Y}_{n-1}\xrightarrow{\pi_{n-1}}\dots \mathcal{Y}_1\xrightarrow{\pi_1}\mathcal{Y}_0=\mathcal{Y}^{min}.
        \end{align*}
        Note that each blowup map corresponds to a disc $D\in \mathcal{D}_{X/Y}.$ For any $i\leq n$, we can view $(\mathcal{{Y}}_i)_s$ as a subset of $\mathcal{Y}^{disc}_s.$ Assume that the statement of Equation \ref{Equation: Mod 2 congruency} is true for all the irreducible components in $\mathcal{Y}_i$ for some $i<n.$ This is true for $\mathcal{Y}_0$, since by definition, we have that $v_{D_0}=v_{\varphi}.$ \par
        Suppose that $\pi_{i+1}\colon \mathcal{Y}_{i+1}\rightarrow \mathcal{Y}_{i}$ is the blowup map corresponding to the disc $D\in \mathcal{D}_{X/Y}.$ Note that the map $\pi_{i+1}$ is given by a blowup centered at a point in $\Gamma_{P(D)}$. To complete the proof, it suffices to show that
        \begin{align*}
            v_{D}-v_{P(D)}=v_{\varphi}(\Gamma_D)-v_\varphi({\Gamma_{P(D)}}) \mbox{\quad (mod 2)}
        \end{align*}
        By the Equation \ref{Equation: vD}, we have the equality $v_D-v_{P(D)}=\vert  D \vert $ $(mod$ $2)$. Let $\mathcal{L}\in Pic(\mathcal{Y}_i)$ and $s\in \mathcal{L}^{\otimes 2} (\mathcal{Y}_i)$ be such that
        \begin{align*}
            N(\mathcal{Y}_i,K(X))=S(\mathcal{Y}_i,\mathcal{L},s).   
        \end{align*}
        By Lemma \ref{Lemma: Extenstion-of-the-line-bundle-to-the-blowup-map}, we see that \[N(\mathcal{Y}_{i+1},K(X))=S(\mathcal{Y}_{i+1},\pi_{i+1}^*(\mathcal{L}),\pi_{i+1}^*s).\]
        Therefore, $v_\varphi(\Gamma_D)$ is equal to the multiplicity of $div(\pi_{i+1}^*s)$ at $\Gamma_D$. By applying \cite[Theorem 9.2.23]{Liu-Alg-Geom}, we get
        \begin{align*}
            v_\varphi({\Gamma(D)})=v_\varphi({\Gamma_{P(D)}})+\sum_{y\in r_{\mathcal{Y}_i}^{-1}(D)}[k(y):F],
        \end{align*}
        where $r_{\mathcal{Y}_i}$ is the reduction map with respect to the proper model $\mathcal{Y}_i$ and we view $D$ as a subset of $Y^\circ.$ Since except possibly twin clusters with depth in $\frac{1}{2}\mathbb{Z}$, all clusters are with integer depth, we see that \[v_\varphi(\Gamma(D))-v_\varphi(P(\Gamma(D))= \vert D \vert\quad \mbox{(mod 2)}.\]
    \end{proof}
    \begin{definition}
        Define $\mathcal{X}^{disc}:=N(\mathcal{Y}^{disc},K(X))$.
    \end{definition}
    \begin{lemma}\label{Lemma:Number-of-ireducible-components-after-normalization}
        Let $D$ be a disc in ${D}_{X/Y}$, and let $\Gamma_D$ denote the corresponding irreducible component in $\mathcal{Y}^{disc}_s$. Then,
        \begin{enumerate}
            \item If $v_D$ is odd, then $D$ is an odd disc and  it is not a defining disc.
            \item There are two irreducible components contained in $\varphi^{-1}(\Gamma_D)$ if and only if $D$ is an \"ubereven disc.
            \end{enumerate}
    \end{lemma}
    \begin{proof}  
        Part (1), follows from \cite[Lemma 5.14]{DDMM} and also Lemma \ref{Lemma:Locally-analytic-structure-of-critical-points}.\par
        Let 
        \begin{align*}
            \mathcal{Y}^{disc}\cong\mathcal{Y}_n\xrightarrow{\pi_{n}}\mathcal{Y}_{n-1}\xrightarrow{\pi_{n-1}}\dots \mathcal{Y}_1\xrightarrow{\pi_1}\mathcal{Y}_0=\mathcal{Y}^{min},
        \end{align*}
        be the decompostion of $\pi^{disc}\colon \mathcal{Y}^{disc}\rightarrow \mathcal{Y}^{min}$ as a sequence of blowups centered at smooth closed points as it is given in \ref{Equation:Decompostion of pi^disc}. Note that each blowup map $\pi_i\colon\mathcal{Y}_i\rightarrow \mathcal{Y}_{i-1}$ corresponds to a disc  $D_i\in \mathcal{D}_{\mathcal{X/Y}},$ therefore, the decomposition above gives an enumeration of discs in $\mathcal{D}_{X/Y}.$ For each $i\leq n$, $\Gamma_{D_i}$ is contained in the exceptional locus of the map $\pi_{i}$. Moreover, following the Construction \ref{Construction:Y^{disc}}, we can view $\Gamma_{D_i}$ as a subset of $\mathcal{Y}_j$, for $j\geq i.$ \par
        By Lemma \ref{Lemma: Extenstion-of-the-line-bundle-to-the-blowup-map}, we can find a collection of line bundles $\mathcal{L}_i\in Pic(\mathcal{Y}_i)$ and sections $s_i\in \mathcal{L}_i^{\otimes 2}$, satisfying the following conditions.
        \begin{itemize}
            \item  For each $i\leq n$, let $\varphi_i\colon \mathcal{X}_i:=(\mathcal{Y}_i,K(X))\rightarrow \mathcal{Y}_i$ be the normalization map, then we have $\mathcal{X}_i\cong S(\mathcal{Y}_i,\mathcal{L}_i,s_i)$.   \item For any $i \leq n-1$, $\mathcal{L}_{i+1}=\pi_{i+1}^*\mathcal{L}_i$, and $s_{i+1}=\pi_{i+1}^*{s_i}.$ 
        \end{itemize}
        Using the above collection $\{(\mathcal{L}_i,s_i)\}$ and applying Lemma \ref{Lemma:Trivialization-of-the-covering-using-vanishing-order}, we get that the number of irreducible components of $(\mathcal{X}_i)_s$ contained in $\varphi_i^{-1}(\Gamma_{D_i})$, is equal to the number of irreducible components of $(\mathcal{X})_s$ contained contained in $\varphi^{-1}(\Gamma_{D_i}).$ If $v_{D_i}$ is odd, then $\Gamma_{D_i}$ is in the branch locus of the covering $\varphi\colon \mathcal{X}^{disc}\rightarrow \mathcal{Y}^{disc}$. So we can assume that $v_{D_i}$ is even. In particular, by Lemma \ref{Lemma:Trivialization-of-the-covering-using-vanishing-order}, $\Gamma_{D_i}$ is not in the branch locus of $\varphi_i$ and $\varphi_i\colon \varphi^{-1}(\Gamma_{D_i})\rightarrow \Gamma_{D_i}$ is a degree two cover of curves over $k_F$. Note that, $\varphi_i^{-1}(\Gamma_{D_i})$ has two irreducible components, if and only if for any critical point $y\in \mathcal{Y}_i$ contained in $\Gamma_{D_i}$,
        \begin{align*}
            Spec(\hat{\mathcal{O}}_{\mathcal{Y}_i,y}\otimes\mathcal{O}_{\mathcal{X}_i}\otimes k_F)
        \end{align*}
        has two irreducible components. Assume that $y$ corresponds to a subdisc $D\subset D_i$ and let $r=\vert D_i\vert$. Given that $\hat{t}_y\in \mathcal{O}_F[[w]]$ is the trivialization of the cover $\varphi_i$ at $y$, Equation \ref{Equation:Local-Trivialization-isomorphism} and Remark \ref{Remark: Polynomial type of the local trivialization} gives that
        \begin{align*}
        \hat{\mathcal{O}}_{\mathcal{Y}_i,y}\otimes\mathcal{O}_{\mathcal{X}_i}\cong \mathcal{O}_F[[w,z]]/(z^2-\hat{t}_y)\cong \mathcal{O}_F[[w,z]]/(z^2-(w^r-\pi_Fg(w)),
        \end{align*}
        for some polynomial $g(w)\in \mathcal{O}_F[w]$ of degree less than $r.$ We note that $Spec(\hat{\mathcal{O}}_{\mathcal{Y}_i,y}\otimes\mathcal{O}_{\mathcal{X}_i}\otimes k_F)$ consists of two irreducible components if and only if $r$ is even. Consequently, $\varphi^{-1}(\Gamma_{D_i})$ contains two irreducibe component of $\mathcal{X}^{disc}_s$ if and only if $D_i$ is \"ubereven.
    \end{proof}
    \begin{proposition}\label{X_disc-is-regular}
        The model $\mathcal{X}^{disc}$ is regular.
    \end{proposition}
    \begin{proof}
        Find $\mathcal{L}\in Pic(\mathcal{Y}^{disc})$ and $s\in \mathcal{L}^{\otimes 2}(\mathcal{Y}^{disc})$ so that $\mathcal{X}^{disc}=S(\mathcal{Y},\mathcal{L},s).$ Then by Lemma \ref{Lemma:Vanishing-order-mod-2-on-Y^{disc}} we see that
        \begin{align}
            div(s)\equiv \sum_{D\in \mathcal{D}_{X/Y}} v_D\Gamma_D + \sum_{y\in B(L)}\overline{\{y\}}\quad \mbox{(mod 2)}.
        \end{align}
        As $\varphi\colon X\rightarrow Y$ satisfies the semi-stability criterion of Theorem \ref{Theorem: Semi-stability-criterion-for 2:1-covers}, no two vertical divisors with odd multiplicity intersect. By Lemma \ref{Lemma:Number-of-ireducible-components-after-normalization}, any disc $D\in \mathcal{D}_{X/Y}$, if $v_D$ is odd, then $D$ is not a defining disc. Therefore, if $v_D$ is odd, the only horizontal components intersecting $\Gamma_D$, are horizontal components corresponding to twins with non-integer depth. Now by applying Lemma \ref{Lemma:Simultaneous-desingularization} we see that $\mathcal{X}^{disc}$ is regular.
    \end{proof}
    As $\mathcal{X}^{disc}$ is regular, it dominates the minimal regular model of $X$ over $F$. We denote this birational map by $\pi^{disc}\colon \mathcal{X}^{disc}\rightarrow \mathcal{X}^{min}$.
    \begin{lemma}\label{Lemma: Singularity types of X^disc}
        Let $x\in \mathcal{X}^{disc}$ be a singular point of $(\mathcal{X}^{disc})_s$. Set $y=\varphi(x).$ Then, the point $y$ is either the intersection point of two irreducible components of $(\mathcal{Y}^{disc})_s$ or, it is a critical point corresponding to a twin cluster with non-integer depth. In both cases, $x$ is a nodal singularity of $(\mathcal{X}^{disc})_s.$ Moreover, $\mathcal{X}^{disc}_s$ is semi-stable after contracting all of its exceptional curves.
    \end{lemma}
    \begin{proof}
        Note that by Lemma \ref{Lemma:Existence of branch separating model}, $\overline{\{B(L)\}}\cap \mathcal{Y}^{disc}_s$ is contained in the smooth locus of $\mathcal{Y}^{disc}_s$. Consequently, we only need to consider the critical points of $\mathcal{\varphi}$ with respect to $\mathcal{Y}^{disc}$, which are contained in the smooth locus $\mathcal{Y}^{disc}_{sm}$ and the intersection points between irreducible components which are contained in the branch locus $\varphi\colon \mathcal{X}^{disc}\rightarrow \mathcal{Y}^{disc}$. 
        \par 
        Assume that $y$ is a critical point. Then, by Construction \ref{Construction:Y^{disc}} and Theorem \ref{Theorem: Semi-stability-criterion-for 2:1-covers}, we get that $m:=\vert (r_{\mathcal{Y}^{disc}})^{-1}(y)\cap B(L))\vert \in \{1,2\}$. If $m=1$, then $\varphi^{-1}(y)$ is a smooth point of $\mathcal{X}^{disc}_s$. If $m=2$, by Remark \ref{Remark: Polynomial type of the local trivialization},  we get the the local trivialization $\hat{t}_y\in \mathcal{O}_F[[w]]$ is given by 
        \[\hat{t}_y=w^2-\pi_F^r.\] 
        for some odd positive integer $r$. Using Equation \ref{Equation:Local-Trivialization-isomorphism}, we have 
        \begin{align*}
            \mathcal{O}_{\mathcal{X}^{disc}}\otimes_{\mathcal{O}_{\mathcal{Y}^{disc}}} \hat{\mathcal{O}}_{\mathcal{Y}^{disc},y}&\cong {\mathcal{O}}_{F}[[z,w]]/(z^2-(w^2-\pi^r))\\
            &\cong \mathcal{O}_F[[u,v]]/(uv- \pi^r),
        \end{align*}
        where the last isomorphism is provided by a change of variables. We prove the remaining statements of this lemma, by following the proof of \cite[Lemma 5.14]{DDMM}. \par
        Assume that $y\in \Gamma_{D_1}\cap \Gamma_{D_2}$. If both of these components are not in the branch locus, then points in $\varphi^{-1}(y)$ are nodal. Assume that $v_{D_1}$ is odd. By Lemma \ref{Lemma:Number-of-ireducible-components-after-normalization}, we get that $D_1$ is an odd disc which is not a defining disc. Assume that $D_1=P(D_2)$ and set $D_0:=P(D_1)$. In the case that $D_2=P(D_1),$ define $D_0$ to be the unique immediate subdisc of $D_1$. For $i=0,2$, $v_{D_i}$ is even and $\varphi^{-1}(\Gamma_{D_i})\cap \mathcal{X}^{disc}_s$ is consists of one irreducible component, as they are not \"ubereven. Define $C_i:=\varphi^{-1}(\Gamma_{D_i})\cap \mathcal{X}^{disc}_s$, for $i\in\{0,1,2\}$. 
        Note that by Lemma \ref{Lemma:Trivialization-of-the-covering-using-vanishing-order}, $C_0$ is an irreducible component of multiplicity 2. Considering the fact that $C_0.\mathcal{X}^{disc}_s=0$ (\cite[Theorem 9.1.21]{Liu-Alg-Geom}, and by applying the projection formula
        \cite[Theorem 9.2.12]{Liu-Alg-Geom}, we see that $C_0.C_1=C_2.C_1=1$. Therefore, $C_1$ intersects other irreducible components of $\mathcal{X}^{disc}_s$ transversally. \par
        This proves that $(\mathcal{X}^{disc}_s)_{red}$ is semi-stable. Note that by \cite[Lemma 4.10]{DDMM}, any component $\Gamma_D$ where $v_{D}$ is odd, intersect other components of $\mathcal{X}^{disc}_s$ in two smooth points contained in distinct reduced irreducible components. Therefore, by \cite[10.3.35]{Liu-Alg-Geom}, after contracting the exceptional curves if $\mathcal{X}^{disc}_s$, singularities of $\mathcal{X}^{min}$ remain nodal.
    \end{proof}
    Using Lemma \ref{Lemma:Locally-analytic-structure-of-critical-points} and  \cite[Proposition 5.12]{DDMM}, we can describe the special fiber $\mathcal{X}^{disc}$.
    \begin{proposition}\label{Proposition:Description-of-minimal-regular-model}
        For any disc $D\in \mathcal{D}_{X/Y}$, the irreducible components of $(\mathcal{X}^{disc})_s$ contained in $\varphi^{-1}(\Gamma_D)$ are given by
        \begin{enumerate}
        \item If $D=D(\textbf{s}_0)$ is the depth zero disc and $\textbf{s}_0$ is not \"ubereven , then $\varphi^{-1}(\Gamma_D)\subset \mathcal{X}^{disc}_s$ consists of one irreducible component intersecting itself for each twin cluster with depth $\frac{1}{2}.$
        \item  If $D=D(\textbf{s}_0)$ is the depth zero disc and $\textbf{s}_0$ is  \"ubereven, then $\varphi^{-1}(\Gamma_D)\subset \mathcal{X}^{disc}_s$ consists of two irreducible components intersecting each other for each twin cluster with depth $\frac{1}{2}$.
        \item If $D=D(\textbf{s})\neq D(\textbf{s}_0)$ for some principal cluster $\textbf{s}\subset B(L)$, and $\textbf{s}$ is \"ubereven, then $\varphi^{-1}(\Gamma_D)$ consists of two smooth curves intersecting each other for any twin clusters $\textbf{s}^{\prime}\leq \textbf{s}$ where $d_{\textbf{s}^{\prime}}=d_\textbf{s}+\frac{1}{2}.$
        \item If $D=D(\textbf{s})\neq D(\textbf{s}_0)$ for some non-\"ubereven cluster $\textbf{s}\subset B$, then $\varphi^{-1}(\Gamma_D)$ contains one irreducible component intersecting itself for each twin cluster $\textbf{s}^{\prime}\leq \textbf{s}$, such that $d_{\textbf{s}^{\prime}}=d_{\textbf{s}}+\frac{1}{2}.$
        \item If $D\cap B(L)=\textbf{s}$ where $\textbf{s}$ is a twin cluster such that $d_\textbf{s}=d_D+\frac{1}{2}$, then $\varphi^{-1}(\Gamma_D)$ consists of  two copies of $\mathbb{P}_{k_F}^1$ intersecting each other at a single point.
        \item if $v_D$ is odd then $\varphi^{-1}(\Gamma_D)$ consists of one irreducible component isomorphic to $\mathbb{P}^1_{k_F}$ and with self intersection -1.
        \item If $D$ does not satisfy any of the conditions given in (1)-(6) and if $\vert D \cap B(L) \vert $ is even (resp. odd) then $\varphi^{-1}(\Gamma_D)$ consists of two irreducible component (resp. one component) isomorphic to $\mathbb{P}^1_{k_F}.$
        \end{enumerate}
        Moreover, all singularities of $\mathcal{X}^{disc}_s$ are transversal and with the exception of type (6) irreducible components all irreducible components of $\mathcal{X}^{min}_s$ are reduced. 
    \end{proposition}
    \begin{proof}
        We split the proof of this proposition to the properties of the special fiber of $\mathcal{X}^{disc}_s$. Statements about the singularities of $\mathcal{X}^{disc}_s$ follows from Lemma \ref{Lemma: Singularity types of X^disc}. Statements about the number of components contained in the pre-image of an irreducible component of $\mathcal{Y}^{disc}_s$ follows from Lemma \ref{Lemma:Number-of-ireducible-components-after-normalization}. Statements about the multiplicities of irreducible components follows from Lemma \ref{Lemma:Trivialization-of-the-covering-using-vanishing-order}, Lemma \ref{Lemma:Vanishing-order-mod-2-on-Y^{disc}} and also part (1) of Lemma \ref{Lemma:Number-of-ireducible-components-after-normalization}. Let $D$ be a disc of any of the type (5)-(7). Then, $\varphi^{-1}(\Gamma_D)\cap \mathcal{X}^{disc}_s$ is smooth. Moreover, applying Riemann-Hurwitz (\cite[7.4.16]{Liu-Alg-Geom}) gives that irrecucible component contained in $\varphi^{-1}(\Gamma_D)$ has genus zero. Since we are assuming $k_F$ is algebraically closed, these components are isomorphic to $\mathbb{P}^1_{k_F}.$\par
        Let $D\in \mathcal{D}_{X/Y}$ be a disc of type $(6)$ and define $C:=\varphi^{-1}(\Gamma_D)\cap \mathcal{X}^{disc}_s$. It remains to show that $C.C=-1.$ Let $\Gamma_{D_1},\Gamma_{D_2}$ be the two irreducible components of $\mathcal{Y}^{disc}_s$ intersecting $\Gamma_D.$ As explained in the proof of Lemma \ref{Lemma: Singularity types of X^disc}, for $i=1,2$, $C_i:=\varphi^{-1}(\Gamma_{D_i})\cap \mathcal{X}^{disc}_s$ consists of a single irreducible component. Furthermore, we have $C_i.C=1$. Now since $C.\mathcal{X}_s=0$ ( \cite[9.1.21]{Liu-Alg-Geom}), and since $C$ has multiplicity 2, we get that $C.C=-1.$  
    \end{proof}
    We end this section by determining the reduction type of $X$ using the the data of the cluster picture $\Sigma_{X/Y}.$ \begin{proposition}\label{Proposition:Reduction-type}
        Assume that $\varphi\colon X\rightarrow Y$ is a degree 2 Galois cover of smooth curves satisfying the conditions given in Assumption \ref{Assumption:Main assumption}. Let $\Sigma_{X/Y}$ be the cluster picture of this covering with respect to $\mathcal{Y}^{min}$. Then,
        \begin{enumerate}
            \item $X$ is semi-stable if and only if $\Sigma_{X/Y}$ satisfies the semi-stability criterion given in Theorem \ref{Theorem: Semi-stability-criterion-for 2:1-covers}.
            \item $X$ has good reduction if and only if $F(B)=F$, $v_\varphi$ is even and $\Sigma_{X/Y}$ consists of a single cluster of depth zero.
            \item $X$ has potentially tame semi-stable reduction if and only if $[F(B):F]$ is coprime to $char$ $k_F.$
            \item  $X$ has potentially tame good reduction over $F$ if and only if $\Sigma_{X/Y}$ consists of a single cluster of depth zero. If $v_{\varphi}$ is odd, then $X$ has good reducion over a degree two extension $L/F.$
        \end{enumerate}
    \end{proposition}
    \begin{proof}
        The first statement is given in Theorem \ref{Theorem: Semi-stability-criterion-for 2:1-covers}. The second statement follows from the description of the special fiber of the minimal regular model of $X$ provided in Proposition \ref{Proposition:Description-of-minimal-regular-model} and the fact that $g(\mathcal{Y}^{min}_s)>0.$ \par
        Let $L$ be a minimal extension of $F$ where $X_L$ is semi-stable. Let $e:=[F(B):F]$ and $\ell:=[L:F].$ By Theorem \ref{Theorem: Semi-stability-criterion-for 2:1-covers}, we observe that $\ell\in \{2e,e\}$ if $e$ is odd, and $\ell\in \{2e,e,\frac{e}{2}\}$ otherwise. In particular, as $char$ $k_F$ is coprime to 2, we have $(char k_F, \ell)=(char$ $k_F, e).$ On the other hand, since we are assuming $k_F$ is an algebraically closed field, $X$ has potentially tame semi-stable reduction if and only if $(char$ $k_F,\ell)=1.$
        \par Assume that $L/F$ is a minimal extension with the property that $X_L$ has good reduction. By part (1), the cluster picture of the cover $\varphi$ over $L$ has a single cluster of depth zero. This implies that the cluster picture associated to $\varphi$ over $F$ also has a single cluster of depth zero. By Theorem \ref{Theorem: Semi-stability-criterion-for 2:1-covers}, if $X$ is not semi-stable, then $v_\varphi$ is odd. For a degree two extension $L/F$, we see that $v_\varphi$ is even over $L$, therefore it has good reduction.
    \end{proof}
    \section{Galois action and the computation of the normalized volume of Hitchin fibers}\label{Section: Galois action}
    Assume that $Y/F$ is a smooth curve, such that $g(Y)>0$ and $Y$ admits a proper smooth model over $F$. Fix a line bundle $\mathcal{L}\in Pic(\mathcal{Y}^{min})$, and let 
    \begin{align*}
                    h:\mathcal{M}_{\mathcal{Y}}^\mathcal{L}&(2,d)\longrightarrow \mathcal{A}_{\mathcal{Y}}^\mathcal{L}(2,d)\\
                &(E,\theta)\longmapsto (Tr(\theta),det(\theta))
    \end{align*}
    denote the Hitchin fibration over the moduli space of rank 2, degree d $\mathcal{L}$-twisted Higgs bundles. For a generic choice of $\alpha,$ spectral correspondence gives an isomorphism 
    \begin{align*}
        h^{-1}(\alpha)\cong Jac(\mathcal{Y}_{\alpha})
    \end{align*}
    where $\mathcal{Y}_{\alpha}\subset Tot(\mathcal{L})$ is a degree two cover of $\mathcal{Y}$ (see Remark \ref{Remark: Spectral correspondence}).
    Define the semi-stable locus $\mathcal{A}_{\mathcal{Y}}^\mathcal{L}(2,d)_{st}\subset \mathcal{A}_{\mathcal{Y}}^\mathcal{L}(2,d)$, to be  the open set consisting of points $\alpha \in \mathcal{A}_{\mathcal{Y}}^\mathcal{L}(2,d)$, where $Y_{\alpha}:=\mathcal{Y}_{\alpha}\times_{Spec(\mathcal{O}_F)} Spec(F)$ admits a semi-stable model over $F$. In this section, we give an algorithm for computing the normalized $p$-$adic$ volume of Hitchin fibers over the semi-stable locus $ \mathcal{A}_{\mathcal{Y}}^\mathcal{L}(2,d)_{st}.$ To do this, we first view $Y_\alpha$ as a degree two Galois cover of $Y$, we denote this cover by $\varphi\colon Y_{\alpha}\rightarrow Y$. The cluster picture, associated to this cover, determines the dual graph of the minimal regular model of $Y_{\alpha}.$ After describing the action of $Frob\in G_F$ on the dual graph $Gr(Y_{\alpha})$ (see Definition \ref{Definition: Dual-Graph}), using the equality given in Equation \ref{Equation: Computation of Tamagawa number}, we can compute Tamagawa number $\vert \upphi(Jac(Y_\alpha)(k_F)\vert$. By Lemma \ref{Lemma: Normalized-volume}, this computation gives the normalized p-adic volume of $h^{-1}(\alpha).$ The missing ingredient in above process, is the description of the Frobenius action on the dual graph $Gr(Y_{\alpha})$, which we explain in this section.
    Note that to construct the cluster picture, we need to consider the base change of a given Galois cover to a choice of a maximally unramified extension (see Remark \ref{Remark: Dual-graph using maximally unramified}). \par Suppose we are given a degree two Galois cover of smooth curves $\varphi\colon X\rightarrow Y$ satisfying the conditions given in Theorem \ref{Theorem: Semi-stability-criterion-for 2:1-covers}. For simplifying the notations, we define $Z:=X_{F^{ur}}$ and $W:=Y_{F^{ur}}.$ By an abuse of notation, we denote the base change of the covering to $F^{ur}$ by $\varphi\colon Z\rightarrow W.$ The action of $Gal(F^{ur}/F)$ on $W$, induces an action on the cluster picture of the covering $\varphi\colon Z\rightarrow W$ by permutation action on the horizontal divisors. This action extends to an action on the collection of admissible discs $\mathcal{D}_{Z/W}$. Moreover, by the universal property of the minimal regular models, the action of $Gal(F^{ur}/F)$ on $W$ extends to an action on $\mathcal{W}^{min}.$ We claim that the action $Gal(F^{ur}/F)$ on $\mathcal{W}^{min}$ extends to an action on $\mathcal{W}^{disc}$. To see this, note that the model $\mathcal{W}^{disc}$ is constructed from $\mathcal{W}^{min}$, by a sequence of blowups centered at Galois orbits of smooth closed points. In particular, the action of $Gal(F^{ur}/F)$ on $\mathcal{W}^{min}$, extends to an action on $\mathcal{W}^{disc}.$ Similarly, $Gal(F^{ur}/F)$ acts on $\mathcal{Z}^{min}$, and this action extends to an action on $\mathcal{Z}^{disc}.$\par
    Let $\varphi\colon \mathcal{Z}^{disc}\rightarrow \mathcal{W}^{disc}$ denote the normalization map. Let $\imath \in Aut(\mathcal{Z}^{disc})$ denote the involution automorphism of $\mathcal{X}^{disc}$, given by the normalization map. Note that the action of $Gal(F^{un}/F)$, commutes with the involution automorphism.\par
    On the other hand, by Remark \ref{Remark: Disk-Component correspondence}, there is a one to one correspondence between discs in $\mathcal{D}_{Z/W}$ and irreducible components of $\mathcal{W}^{disc}_s$. Furthermore, this correspondence is $Gal(F^{ur}/F)$ equivariant. Hence, the action of $Gal(F^{ur}/F)$ on the cluster picture of the covering $\varphi$, determines the induced action of $Gal(F^{ur}/F)$ on the irreducible components of $\mathcal{W}^{disc}_s.$ Therefore, by the discussion above, the action of $Gal(F^{ur}/F)$ on the cluster picture $\Sigma_{Z/W}$, determines the Galois action on involution orbits of the irreducible components of $\mathcal{Z}^{disc}_s$. \par
    Given a disc $D\in \mathcal{D}_{Z/W},$ let $\Gamma_D$ denotes the corresponding irreducible component in in $\mathcal{W}^{disc}_s$. We denote the irreducible components of $\mathcal{Z}^{disc}_s$, contained in $\varphi^{-1}(\Gamma_D)$ by $\Gamma_D^{\prime \pm}$. If there is one component contained in $\varphi^{-1}(\Gamma_D)$, then $\Gamma_D^{\prime +}=\Gamma_D^{\prime -}$. 
    Let 
        \begin{align*}         \mathcal{W}^{disc}\cong\mathcal{W}_n\xrightarrow{\pi_{n}}\mathcal{W}_{n-1}\xrightarrow{\pi_{n-1}}\dots \mathcal{W}_1\xrightarrow{\pi_1}\mathcal{W}_0=\mathcal{W}^{min},
    \end{align*}
    be a decomposition of the birational map $\pi\colon \mathcal{W}^{disc}\rightarrow \mathcal{W}^{min}$ into a sequence of blowups centered at smooth closed points. By Remark \ref{Remark: Disk-Component correspondence}, each blowup map corresponds to a disc $D\in \mathcal{D}_{Z/W}.$ By Lemma \ref{Lemma: Extenstion-of-the-line-bundle-to-the-blowup-map}, we can find collection of line bundles $\mathcal{L}_i$ and sections $s_i\in \mathcal{L}^{\otimes 2}$, satisfying:
    \begin{itemize}
        \item  For each $i\leq n$, let $\varphi_i\colon \mathcal{Z}_i:=(\mathcal{W}_i,K(Z))\rightarrow \mathcal{W}_i$ be the normalization map, then we have $\mathcal{Z}_i\cong S(\mathcal{W}_i,\mathcal{L}_i,s_i)$. 
        \item For any $i \leq n-1$, $\mathcal{L}_{i+1}=\pi_{i+1}^*\mathcal{L}_i$, and $s_{i+1}=\pi_{i+1}^*{s_i}.$ 
    \end{itemize}
    Assume that $\pi_i\colon\mathcal{W}_i\rightarrow \mathcal{W}_{i-1}$ be the blowup corresponding the a defining disc $D(\textbf{s})$ where $\textbf{s}$ is even. Assume that $\pi_i\colon \mathcal{W}_i\rightarrow \mathcal{W}_{i-1}$, is the blowup centered at $w\in \Gamma(P(D)).$ Let $\hat{t}_w\in \hat{\mathcal{O}}_{\mathcal{W}_{i-1},w}$ be the local trivialization of the cover $\varphi_i\colon \mathcal{Z}_i:=N(\mathcal{W}_i,K(W))\rightarrow \mathcal{W}_i$ at  the point $w$. Since $\textbf{s}$ is even, we have that $\hat{t}_w$ is square modulo the maximal ideal of $\hat{\mathcal{O}}_{\mathcal{W}_{i-1},w}$.  We fix $\theta_s\in k(w)=\overline{k_F}$ to be a choice of a square of $\hat{t}_w$. Note that this induces a function
    \begin{align*}
        \theta\colon\{\mbox{\textbf{s}} \in \Sigma_{Z/W}\mid  \mbox{ \textbf{s} is even}&\}\longrightarrow \overline{k}_F,\\
        &\textbf{s}\longmapsto \theta_\textbf{s}.
    \end{align*}
    Now for any even cluster $\textbf{s}\in \Sigma_{Z/W}$, define  
    \begin{align*}
        \epsilon_s\colon Gal(F^{ur}/&F)\longrightarrow \{+1,-1\}\\
        &\sigma\longmapsto \frac{\sigma(\theta_{\textbf{s}})}{\theta_{\sigma(\textbf{s})}}.
    \end{align*}
    Furthermore, for any principal cluster $\textbf{s}\in \Sigma_{Z/W},$ where $\textbf{s}$ in not even, we set $\epsilon_\textbf{s}(\sigma)=1$ for all $\sigma\in Gal(F^{ur}/F).$\par
    Note that for any \"ubereven cluster,  $\theta_\textbf{s}$ determines a choice of an irreducible component in the pre-image of $\varphi^{-1}(\Gamma_{\textbf{s}})$. As we discussed in the beginning of this section, the action of $Gal(F^{ur}/F)$ on the cluster picture $\Sigma_{Z/W}$ determines the action of $Gal(F^{ur}/F)$ on the involution orbits of $\mathcal{Z}^{disc}.$ Given an element of the Galois group $\sigma \in Gal(F^{ur}/F)$, the function $\epsilon_\textbf{-}(\sigma)$ determines whether $\sigma$ permutes the components inside the involution orbits. Now the following theorem, determines the structure of the dual graph $Gr(Z)$, and furthermore describes the action of $Gal(F^{ur}/F)$ on the dual graph. The description of the action of $Gal(F^{ur}/F)$ on $Gr(Z)$ is similar to the description given in \cite[Theorem 8.5]{DDMM}.
    \begin{theorem}\label{Theorem: structure of the dual graph and Galois action}
        Assume that $\varphi\colon X\rightarrow Y$ is a degree two Galois cover of smooth curves over $F$, with $Y$ admitting a smooth proper model and $X$ admitting a semi-stable model. Let $Z=X_{F^{ur}}$ and $W=Y_{F^{ur}}$. Then, the dual graph of $Z$ is structured as follow. Corresponding to each principal cluster in $\textbf{s}\in \Sigma_{Z/W}$ there are one or two vertices in $Gr(Z).$ There are two vertices if $\textbf{s}$ is \"ubereven. We denote this vertices by $v_{\textbf{s}}^{\pm}$, and we set $v_{\textbf{s}}^+=v_{\textbf{s}^{\prime}}^-$ if $\textbf{s}$ is not \"ubereven. These vertices are linked by chain of edges described below.
        \begin{table}[h]
        \begin{tabular}{|l|l|l|l|l|}
            \hline
            name                      & from                        & To                 & length                                    & conditions                                                                   \\ \hline
            $L_{\textbf{s}^{\prime}}$ & $v_{\textbf{s}^{\prime}}$   & $v_\textbf{s}$     & $\frac{1}{2}\delta_{\textbf{s}^{\prime}}$ & $\textbf{s}^{\prime}<\textbf{s}$, both principal, $\textbf{s}^{\prime}$ odd  \\ \hline
            $L_{\textbf{s}^{\prime}}$ & $v_{\textbf{s}^{\prime}}^+$ & $v_\textbf{s}^+$   & $\delta_{\textbf{s}^{\prime}}$            & $\textbf{s}^{\prime}<\textbf{s}$, both principal, $\textbf{s}^{\prime}$ even \\ \hline
            $L_{\textbf{s}^{\prime}}$ & $v_{\textbf{s}^{\prime}}^-$ & $v_\textbf{s}^-$   & $\delta_{\textbf{s}^{\prime}}$            & $\textbf{s}^{\prime}<\textbf{s}$, both principal, $\textbf{s}^{\prime}$ even \\ \hline
            $L_\textbf{t}$            & $v_\textbf{s}^{-}$          & $v_\textbf{s}^{+}$ & $2\delta_\textbf{t}$                      & $\textbf{s}$ principal, $\textbf{t}<\textbf{s}$ twin                         \\ \hline
        \end{tabular}
        \end{table}
        Furthermore, given $\sigma\in Gal(F^{ur}/F)$, $\sigma$ acts on $Gr(Z)$ as follow.
    \begin{enumerate}
        \item $\sigma(v_{\textbf{s}^{\prime}}^{\pm})=v_{\sigma(\textbf{s})}^{\pm\epsilon_\textbf{s}(\sigma)}$,
        \item $\sigma(L_{\textbf{s}^{\pm}})=L_{\sigma(\textbf{s})}^{\pm\epsilon_\textbf{s}(\sigma)}$
        \item For $\textbf{t}$ twin, $\sigma(L_\textbf{t})=\epsilon_{\textbf{t}}(\sigma)L_{\sigma(\textbf{t})},$ where $-L_{\textbf{t}}$ indicates the negative orientation of this chain of edges.
    \end{enumerate}
    \end{theorem}
    \begin{proof}
        The proof of this theorem is similar to the proof of \cite[Theorem 8.5]{DDMM} and it follows from Propostion \ref{Proposition:Description-of-minimal-regular-model}.
    \end{proof}


\begin{thebibliography}{9}
\bibitem{AW}
K. Arzdorf and S. Wewers,
\textit{Another proof of the Semistable Reduction
Theorem}, https://arxiv.org/pdf/1211.4624.pdf
\bibitem{BETTS}
A. Betts
\textit{Variation of Tamagawa numbers of Jacobians of hyperelliptic curves with semistable reduction}, To appear in Journal of Number Theory.
\bibitem{BLR}
S. Bosch, W. L\"utkebohmert, and Michel Raynaud,
\textit{N\'eron models}, vol. 21, Springer
Science \& Business Media, 2012.
\bibitem{BL-unif}
S.Bosch, W. L\"utkebohmert,
\textit{Stable reduction and uniformization of abelian varieties I}, to appear in Math. Ann.
\bibitem{DM}
P. Deligne, D. Mumford, 
\textit{The irreducibility of the space of curves of given genus}, Publ. Math. IHES,
Tome 36 (1969), 75–109.
\bibitem{DDMM} 
T.Dokchitser, V. Dokchitser, C. Maistret and A. Morgan,
\textit{Arithmetic of hyperelliptic curves over local fields,} To appear in Mathematische Annalen.
\bibitem{DDMM2} 
T.Dokchitser, V. Dokchitser, C. Maistret and A. Morgan,
\textit{Semistable types of hyperelliptic curves,} preprint (2018)
\bibitem{OF}
O. Faraggi, S. Nowell
\textit{
Models of Hyperelliptic Curves with Tame Potentially Semistable Reduction,}
Trans. London Math. Soc., 7(1):49–95, 2020.
\bibitem{GM}
M. Groechenig, M. McBreen,
\textit{Hypertoric Hitchin systems and Kirchhoff polynomials,}
https://arxiv.org/pdf/2001.11084.pdf(2020)
\bibitem{LH}
Lars Halvard Halle
\textit{Stable reduction of curves and tame ramification
}
\bibitem{LP}
\textit{L. Levine and J. Propp, What is a Sandpile?,} Notices of the AMS volume 57 number 8 (2010),
pages 976–979.
\bibitem{LL99}
Q.Liu, D. Lorenzini
\textit{Models of curves and finite covers}
\bibitem{Liu-Alg-Geom}
Q. Liu,
\textit{Algebraic geometry and arithmetic curves,} Oxford Univ. Press, 2002.
\bibitem{L90}
D. Lorenzini
\textit{Dual Graphs of Degenerating Curves,} Mathematische Annalen 287 (1990), 135 - 150.
\bibitem{Mats-Comalg}
H. Matsumura
\textit{Commutative Algebra (2nd edn),}
Commutative Algebra (2nd edn), Benjamin/Cummings, 1980.
\bibitem{Raynaud}
M. Raynaud
\textit{p-groupes et r\'eduction semi-stable des courbes,} in The Grothendieck
Festschrift, Vol. III, Prog. Math., 88, Birkh\"auser, 1990, 179–197.
\bibitem{Ruth}
J. R\"uth,
\textit{Models of curves and valuations,} https://d-nb.info/1072464500/34(Thesis)
\bibitem{Srin15}
P. Srinivasan,
\textit{Conductors and minimal discriminants of hyperelliptic
curves with rational Weierstrass points}
https://arxiv.org/abs/1508.05172(2015)
\bibitem{Weil}
A. Weil,
\textit{Adeles and algebraic groups,} vol. 23, Springer Science \& Business Media, 2012.
\end{thebibliography}
\end{document}